\documentclass[reqno,a4paper]{amsart}
\usepackage{enumitem}
\setlist[enumerate]{label=\emph{(\roman*)}}

\newtheorem{theorem}{Theorem}
\newtheorem{lemma}{Lemma}
\newtheorem{proposition}{Proposition}
\theoremstyle{remark}
\newtheorem{remark}{Remark}
\numberwithin{remark}{section}
\numberwithin{proposition}{section}
\numberwithin{lemma}{section}
\numberwithin{equation}{section}

\newcommand\R{\mathbb{R}}
\newcommand\C{\mathbb{C}}

\newcommand\ii{\textnormal{i}}
\newcommand\ee{\textnormal{e}}
\newcommand\ext{\textnormal{ext}}
\newcommand\itt{\textnormal{int}}
\newcommand\Lp{L_+}
\newcommand\Lm{L_-}
\newcommand\Hp{{\mathcal I}_+}
\newcommand\Hm{{\mathcal I}_-}
\newcommand\Np{\mathcal N_+}
\newcommand\Nm{\mathcal N_-}
\newcommand\NLm{N_-}
\newcommand\NLp{N_+}
\newcommand\tNLm{\tilde N_-}
\newcommand\tNLp{\tilde N_+}
\newcommand\FFm{F_-}
\newcommand\FFp{F_+}

\DeclareMathOperator\Ai{\bf A \kern-0.1em i}
\DeclareMathOperator\BB{\bf B}

\title[Self-similar blow-up profile]{Self-similar blow-up profiles for slightly supercritical nonlinear Schr\"odinger equations}

\author[Y. Bahri]{Yakine Bahri}
\address{Department of Mathematics and Statistics, University of Victoria, Canada}
\email{ybahri@uvic.ca}

\author[Y. Martel]{Yvan Martel}
\address{CMLS, Ecole polytechnique, CNRS, Institut Polytechnique de Paris, 91128 Palaiseau Cedex, France}
\email{yvan.martel@polytechnique.edu}

\author[P. Rapha\"el]{Pierre Rapha\"el}
\address{DPMMS, Centre for Mathematical Sciences, University of Cambridge, Wilberforce road, Cambridge CB3 0WA, U.K.}
\email{pr463@cam.ac.uk}


\begin{document}

\begin{abstract}
We construct radially symmetric self-similar blow-up profiles for the mass supercritical nonlinear Schr\"odinger equation $\ii \partial_t u + \Delta u + |u|^{p-1}u=0$ on $\R^d$,
close to the mass critical case and for any space dimension $d\ge 1$. These profiles bifurcate from the ground state solitary wave. The argument relies on the classical matched asymptotics method suggested in \cite{SulemSulem} which needs to be applied in a degenerate case due to the presence of exponentially small terms in the bifurcation equation related to the log-log blow-up law observed in the mass critical case.
\end{abstract}

\maketitle

\section{Introduction}

\subsection{The energy subcritical problem}
We consider the nonlinear Schr\"odinger (NLS) equation
\begin{equation}\label{NLS}
\left\{\begin{aligned}
&\ii\partial_t u + \Delta u +|u|^{p-1}u =0, \quad (t,x)\in\R\times\R^d\\
&u_{|t=0}=u_0, \quad x\in \R^d
\end{aligned}\right.
\end{equation}
in any space dimension $d\geq 1$.

\medskip

\noindent\emph{-- Local Cauchy theory}. For energy subcritical nonlinearities, \emph{i.e.}
under the restriction
 \[1<p<2^*-1= \frac{d+2}{d-2} \quad \mbox{when} \ d\ge 3\]
 (no restriction on $p>1$ when $d=1,2$), the Cauchy problem for \eqref{NLS} is well-posed in the energy space $H^1(\R^d)$: for any $u_0\in H^1$, there exists a unique maximal solution 
 $u\in \mathcal C([0,T),H^1)$ of \eqref{NLS} (see \cite{Caz,GV}).
 Moreover, if the maximal time of existence $T$ is finite, then $\lim_{t\uparrow T} \|u(t)\|_{H^1}=\infty$.
 
 \medskip
 
\noindent\emph{-- Conservation laws}. Such $H^1$ solutions satisfy the conservation of mass and energy: 
\begin{align*}
 \|u(t,\cdot)\|_{L^2}&=\|u_0\|_{L^2},\\
 E(u(t,\cdot))& =\frac 12\|\nabla u(t,\cdot)\|_{L^2}^2-\frac1{p+1}\|u(t,\cdot)\|_{L^{p+1}}^{p+1}=E(u_0).
 \end{align*} 
 
\noindent\emph{-- Scale invariance and critical space}. The scaling
\begin{equation*}
u_\lambda(t,x)=\lambda ^{\frac{2}{p-1}}u(\lambda^2 t,\lambda x), \quad \lambda>0,
\end{equation*}
acts on the space of solutions by leaving the critical Sobolev norm invariant
\[\|u_\lambda(t,\cdot)\|_{\dot{H}^{s_c}}=\|u(\lambda^2t,\cdot)\|_{\dot{H}^{s_c}}\quad
 \mbox{for}\quad s_c=\frac d2-\frac 2{p-1}.\]
 
 \noindent\emph{-- Global existence \emph{versus} blow up}. On the one hand, for the mass subcritical case, \emph{i.e.} $s_c<0$, the conservation of mass and energy combined with the  Cauchy theory ensure from classical arguments 
 (the Gagliardo-Nirenberg inequality, see \emph{e.g.}~\cite{Caz}) that all $H^1$ solutions are global and bounded in $H^1$. On the other hand, for the intercritical case $0\leq s_c< 1$, a consequence of  the classical virial law
\[\frac{d^2}{dt^2}\int|x|^2|u(t,x)|^2dx\le 4d(p-1) E(u_0)\]
implies that any negative energy initial data  with finite variance ($xu_0\in L^2$) yield  blowup in finite time.
By such a contradiction argument, the only   qualitative information  $\|u(t)\|_{H^1} \geq C(T-t)^{-\frac 12 (1-s_c)}$
on the blow-up solution $u$ close to the blow-up time $T$ is provided by the Cauchy theory (see \emph{e.g.}~\cite{Caz}).
 
\medskip

\noindent\emph{-- Ground state.} For $1<p<2^*-1$, the existence, uniqueness and further properties of the ground state solitary wave $Q$, positive radial $H^1$ solution of
\[
\Delta Q - Q + Q^p=0 \quad \text{on $\R^d$,}
\]
are well-known. See \emph{e.g.} \cite[\S4.2]{SulemSulem} and \cite{Caz}.
The function $Q$ is seen as a function of $r=|x|\geq 0$, and it is standard to check that it satisfies, for two constants $\kappa, C>0$ depending on the dimension $d$, for all $r\geq 1$,
\begin{equation}\label{on:Q}
\left|Q(r) - \kappa r^{-\frac{d-1}2} \ee^{-r} \right|+
\left|Q'(r) + \kappa r^{-\frac{d-1}2} \ee^{-r}\right|\leq C r^{-\frac{d+1}2} \ee^{-r}.
\end{equation}

\subsection{Description of blowup} Let us recall the main known results concerning the qualitative description of blow-up solutions of (NLS)  for $0\leq s_c<1.$
 
\medskip
 
\noindent\emph{-- Mass critical case $s_c=0$.} 
In this case, the \emph{stable blow-up regime} formally predicted in \cite{LPSS} corresponds to a log-log deviation from self-similarity, with rate of concentration
\begin{equation}
\label{log-logaws}
\lambda(t)\sim \sqrt{\frac{T-t}{\log |\log (T-t)|}} \quad \mbox{as $t\sim T^-$.}
\end{equation} 
The series of works \cite{MR1,MeRa04,MR3,MeRa07,P} provides a complete description of such singularity formation for initial data near the ground state solitary wave $Q$.
The corresponding solutions blow up by concentrating the  profile $Q$
\[u(t,x)\sim \frac{e^{\ii\gamma(t)}}{\lambda^{\frac{2}{p-1}}(t)}Q\left(\frac{x-x(t)}{\lambda(t)}\right)\]
at a blow-up rate $\lambda(t)$ satisfying \eqref{log-logaws} which is \emph{not} self-similar 
since \[\lim_{t\uparrow T}\frac{\lambda(t)}{\sqrt{T-t}}=0.\]
Such behavior is often called type II blow up. 
Other (unstable) blow-up solutions, obtained through the pseudo-conformal transform enjoy a distinct blow-up rate
$\|u(t)\|_{H^1}\sim (T-t)^{-1}$.
We refer to \cite{MRtwo} for multi-bubbles blow-up scenario and to \cite{MMRkdv3,MMRkdv2,MMRkdv1} for the case of the mass critical generalized KdV equation, which displays a similar critical structure.

\medskip

\noindent\emph{-- Intercritical case $0<s_c<1$.}
In contrast with the mass critical case, 
on the basis of formal arguments and numerical simulations, the existence of blow-up bubbles with self-similar concentration rate $\lambda(t)\sim \sqrt{T-t}$ is conjectured in the intercritical case $0<s_c<1$; see~\cite{SulemSulem2,Zak} and references therein.
 More precisely, given $T\in \R$ and~$b>0$, the ansatz
\begin{equation*}
u(t,x)=\frac1{(2b(T-t))^{\frac 1{p-1}}} \exp\left(-{\ii}\frac{\log (T-t)}{2b}\right) \Psi\left(\frac{x}{(2b(T-t))^{\frac 12}}\right)
\end{equation*}
maps equation \eqref{NLS} exactly onto the time independent problem
\begin{equation}\label{eq:Qb}
\Delta \Psi -\Psi + \ii b \left(\frac 2{p-1} \Psi +x\cdot \nabla \Psi\right) + |\Psi|^{p-1}\Psi = 0\quad \mbox{on $\R^d$}.
\end{equation}
The existence and stability of such type I blow-up regimes $\lambda(t)\sim \sqrt{T-t}$ have been proved in \cite{MRS} for (NLS) with $0<s_c\ll 1$, using a deformation argument of the stability analysis performed in \cite{MR1,MR3} for the mass critical case. However, this approach cannot address the sharp description of the singularity and the local asymptotic stability of the blow-up profile which is observed numerically in~\cite{SulemSulem2}. The missing piece in the analysis is precisely the proof of existence of finite energy solutions to equation \eqref{eq:Qb},
and the determination of the spectral properties of the associated linearized operators. We also mention that for the full intercritical range $0<s_c<1$, completely different non self-similar type II blow-up solutions have been constructed; see~\cite{FGW,MRSring}.

\subsection{Statement of the result} The aim of the present paper is to complete the first step towards a complete description of self-similar blowup, by constructing rigorously finite energy solutions of the stationary self-similar equation \eqref{eq:Qb}.
\begin{theorem}[Existence of a finite energy self-similar profile for $0<s_c\ll1 $]
\label{th:1}
Let $d\geq 1$. Let $p_*$ be the mass critical exponent
\[p_*=1+\frac 4d.\] 
There exists $\epsilon>0$ such that for any $p$ satisfying
\begin{equation}\label{on:p2}
0<p- p_* <\epsilon,
\end{equation}
there exist $b=b(p)>0$ and a non zero radially symmetric solution $\Psi$ to \eqref{eq:Qb}
\[\Psi \in \dot H^1(\R^d)\cap \mathcal C^2(\R^d),\quad E(\Psi)=0.\]
Moreover, it holds, as $p\downarrow p_*:$
\begin{enumerate}
\item[\rm 1.]\emph{Law for the nonlinear eigenvalue:} 
$b=b_{s_c} (1+o(1))$ 
 where $b_{s_c}$ is defined by
 \begin{equation}
 \label{ffinionnonoe}
 s_c= \frac{\kappa^2 }{N_c}b_{s_c}^{-1}\exp\left(-\frac{\pi}{b_{s_c}}\right),\quad 
 N_c = \int_0^\infty Q^2(r) r^{d-1} dr.
 \end{equation}
\item[\rm 2.]\emph{Bifurcation from the soliton profile:}
$\|\Psi - Q\|_{\dot H^1}=o(1).$
\item[\rm 3.]\emph{Non oscillatory behavior for the outgoing wave:}
\begin{equation*}
 \lim_{r\to \infty} r^{\frac 2{p-1}} |\Psi(r)| = \rho_{s_c} (1+o(1)) ,\quad
\limsup_{r\to \infty} r^{\frac {p+1}{p-1}}|\Psi'(r)|<\infty,
\end{equation*}
where
\begin{equation*}
\rho_{s_c}= ( 2 N_c)^{\frac 12} s_c^{\frac 12} =
\sqrt{2} \kappa b_{s_c}^{-1} \exp\left(-\frac{\pi}{2 b_{s_c}}\right).
\end{equation*}
\end{enumerate}
\end{theorem}

\noindent\textbf{Comments on Theorem~\ref{th:1}.}
\medskip

\noindent\emph{1. Matched asymptotics and the {\rm log-log} law}. Theorem \ref{th:1} and the computation of the asymptotics of the nonlinear eigenvalue \eqref{ffinionnonoe} as $p\downarrow p_*$ is explicitly conjectured in \cite{SulemSulem2}. This law is deeply connected to the log-log law \eqref{log-logaws}. The formal argument in \cite[Chapter~8]{SulemSulem}, see also \cite{KL, LPSS,LPSS1,LPSS2, RK, SulemSulem2}, is performed on the \emph{near critical dimension problem} for $p=3$ and $d\downarrow 2:$
\[
\ii\partial_t u + \partial_{rr} u + (d-1) \frac{\partial_r u}{r} +|u|^2 u=0.
\]
For equation \eqref{eq:Qb}, following the classical \emph{matched asymptotic} approach, we aim at finding directly the law $b(p)$ such that we can \emph{glue} the non oscillatory outgoing solution at $r\to +\infty$ which has finite energy, with the smooth radially symmetric solution which emanates from the origin in space as a small deformation of the ground state $Q$. A similar strategy, in a situation where it was simpler to implement, has been used to construct self-similar blow-up profiles for the energy supercritical nonlinear heat equation; see \cite{bizon, CRS}.
See~\cite{JP,KW,PS,Zak} for studies of the asymptotic behavior of self-similar (NLS) solutions.
We also refer to~\cite{BCR,Budd,YRZ} for   formal and numerical investigations on multiple bumps solutions.

\medskip

\noindent\emph{2. Global bifurcation}. 
For the slightly mass supercritical generalized Korteweg-de Vries equation, which displays a similar structure,
an even larger set of profiles is constructed in the very nice work \cite{Koch}. The author uses an abstract Lyapunov-Schmidt bifurcation argument which is more general since it applies to PDE as well, while gluing is an ODE tool. Conceptually, the approach of \cite{Koch} could be applied to prove Theorem~\ref{th:1}. However, 
let us stress that the presence of exponential smallness in the bifurcation law \eqref{ffinionnonoe}  makes the analysis delicate  (compare with Theorem~2 in \cite{Koch}).
 In such a context, one advantage of the gluing method is to make the   dominant terms and the matching procedure  appear more naturally and explicitly. It also  allows a flexible functional framework in order to treat separately the neighborhood of the origin and the neighborhood of $\infty$. A classical difficulty is to analyse the behavior of the solution near the turning point $r=\frac 2b$, and here we shall adapt the \emph{quantitative} WKB approach (see e.g. \cite{Fedoryuk}) as very nicely explained in \cite{Godet}.

\medskip

We see two main open problems in the continuation of the present work. First, understand the dynamical properties of (NLS) in the vicinity of the profiles provided by Theorem \ref{th:1} and prove that these profiles are asymptotic attractors of the flow after renormalization. 
%This will be addressed in a forthcoming work. 
Second, extend such results to the full intercritical range $0<s_c<1$, which is a challenging problem.

\subsection{Outline of the proof}
We outline the strategy of the proof of Theorem~\ref{th:1} (see the extended formulation in Theorem \ref{th:2}, Section~\ref{S:matching}).

\medskip

\noindent{\bf step 1.} One dimensional equations and turning point. 
Looking for a radial solution $\Psi\in \dot H^1(\R)$ of \eqref{eq:Qb}, we change variables
\begin{equation}\label{PsiP}
\Psi(x) = \exp\left(-\ii\frac{br^2}4\right) P(r),\quad r=|x| \geq 0.
\end{equation}
The equation for $P:[0,\infty)\to \C$ becomes
\begin{equation*}\left\{\begin{aligned}
&P''+\frac{d-1}r P' + \left( \frac{b^2r^2}4-1-\ii bs_c\right) P + |P|^{p-1}P=0,\quad r>0,\\
&P'(0)=0.
\end{aligned}\right.\end{equation*}
Note that the condition \eqref{on:p2} for $\epsilon>0$ small is equivalent to assuming that $s_c>0$ is small.
More generally, we will consider the equation
\begin{equation}\label{eq:Pb}\left\{\begin{aligned}
&P''+\frac{d-1}r P' + \left( \frac{b^2r^2}4-1-\ii b \sigma \right) P + |P|^{p-1}P=0,\quad r>0,\\
&P'(0)=0
\end{aligned}\right.\end{equation}
for any $p\geq p_*$ and any $\sigma>0$ small enough, not necessarily related (see the general existence result Theorem~\ref{th:2}
in Section~\ref{S:matching}).
The problem of existence of a suitable solution $P$ can be seen as a nonlinear spectral problem, where only one specific value of $b$ (depending on $\sigma$ and $p$)
provides an admissible solution $P$ close to $Q$, \emph{i.e.} such that the corresponding $\Psi$ belongs to $\dot H^1(\R)$, see \cite[Remark p.135]{SulemSulem}.
Setting
\[
P(r)=r^{-\frac{d-1}2}U(r),
\]
the nonlinear equation \eqref{eq:Pb} will also be considered under the following form
\begin{equation}\label{eq:U}
U''+ \left( \frac{b^2r^2}4-1-\frac{(d-1)(d-3)}{4r^2} -\ii b\sigma\right) U+r^{-\frac12(d-1)(p-1)} |U|^{p-1} U=0
\end{equation}
which makes the turning point $r\sim\frac{2}{b}$ apparent.

\medskip

\noindent{\bf step 2.} A priori control of the free parameters. We consider $b$ in the interval
\begin{equation}\label{on:b}
b\in \left[b_\sigma -\frac 12 b_\sigma^{\frac {13}6}, b_\sigma +\frac 12b_\sigma^{\frac {13}6}\right],
\end{equation}
where $b_\sigma>0$ is defined by the relation
\begin{equation}\label{on:sigma}
\sigma = \frac{\kappa^2 }{N_c}b_\sigma^{-1}\exp\left(-\frac{\pi}{b_\sigma}\right)\quad\mbox{where}\quad
N_c= \int_0^\infty Q^2(r) r^{d-1} dr.
\end{equation}
Note that using \eqref{on:b}, one has
\begin{equation}\label{on:bb}
\left|\exp\left(\frac \pi{b}-\frac{\pi}{b_\sigma}\right)-1\right|
\leq C\left|\frac 1{b}-\frac{1}{b_\sigma}\right|
\leq C b_\sigma^{\frac 16},\quad
\sigma \leq C b^{-1} \exp\left(-\frac{\pi}{b}\right).
\end{equation}
Additional free parameters denoted by $\rho>0$, $\gamma\in \R$ and $\theta\in \R$ will be needed in the construction, under the following constraints
\begin{align}
&\rho\in \left[\frac 12 \rho_\sigma , \frac 32 \rho_\sigma \right],
\label{on:rho}\\
&\gamma\in \left[-\frac 12\gamma_\sigma, \frac 12 \gamma_\sigma\right],
\label{on:gamma}\\
&\theta\in \left[-\frac 12\theta_\sigma,\frac 12 \theta_\sigma\right],
\label{on:theta}
\end{align}
where $\rho_\sigma>0$, $\gamma_\sigma>0$ and $\theta_\sigma>0$ are fixed as follows:
\begin{align*}
\rho_\sigma & =\sqrt{2} N_c^{\frac 12} \sqrt{\sigma} = \sqrt{2} \kappa b_\sigma^{-\frac 12}\exp\left(-\frac{\pi}{2b_\sigma}\right)
,\\
\gamma_\sigma & = b_\sigma^{\frac 16} \exp\left(-\frac2{\sqrt{b_\sigma}}\right) ,\\
\theta_\sigma & = b_\sigma^{\frac 16} \exp\left(-\frac{\pi}{b_\sigma}\right)\exp\left(\frac2{\sqrt{b_\sigma}}\right).
\end{align*}
Since $\sigma$ is to be taken small, we work under the following smallness conditions,
 tacitly used throughout the article
\begin{equation}\label{SM}
0<\sigma\ll1,\quad
0<b\ll 1,\quad 0<\rho \ll 1,\quad |\gamma|\ll 1,\quad |\theta|\ll 1.
\end{equation}

\noindent{\bf step 3.} Three regimes and the matching. 
We decompose $[0,\infty)$ into three regions $I$, $J$ and $K$, defined as follows
\begin{gather*}
r_K=b^{-\frac 12},\quad r_J=2b^{-1},\quad r_I=b^{-2},\\ 
K=[0,r_K],\quad J=[r_K,r_I],\quad I=[r_I,\infty).
\end{gather*}
First, we construct a family of solutions on $I$ whose asymptotics as $r\to \infty$ is admissible,
see~\eqref{eq:th2.1}-\eqref{eq:th2.2}. The free parameter $\rho$ is related to the amplitude of each solution - see Proposition~\ref{pr:UonI}. Second, we extend this family of solutions to the region $J\cup I$ including the \emph{turning point} (this is the point $r$ close to $r_J$ where the real part of the coefficient of $U$ in \eqref{eq:U} vanishes) - see Proposition~\ref{pr:Uext}.
We obtain a family of solutions $P_\ext$ on $J\cup I$ whose specific form at $r_K$ is given by~\eqref{Pext1}-\eqref{Pext4}. Here we shall be particularly careful when tracking exponentially small terms. 
Last, the resolution of the equation in the interval $K$ involves the construction of an approximate solution
close to $Q$. For this, we use the well-known properties of the linearized operator (see \eqref{def:Lp} and \eqref{def:Lm}) around $Q$. See Proposition~\ref{pr:Pint}. In this step, we introduce two additional free parameters: $\gamma$ related to a zero direction of the linearized operator, and $\theta$ related to the phase invariance of equation~\eqref{eq:Pb}.

\medskip

\noindent A key point to match $P_\ext$ and $P_\itt$ is that the general forms obtained for the solutions $P_\ext$, $P_\itt$ 
and their derivatives $P_\ext'$, $P_\itt'$ coincide at $r=r_K$; compare \eqref{Pext1}-\eqref{Pext4} and \eqref{Pint1}-\eqref{Pint4}.
Thus, one only needs to adjust the four free small parameters $b$, $\rho$, $\gamma$ and $\theta$ to exactly match $P_\ext$ and $P_\itt$.
This is done using the Brouwer Fixed-Point Theorem and the continuous dependence of the solutions $P_\ext$ and $P_\itt$ in the free parameters $b$, $\rho$, $\gamma$ and $\theta$
(see Section~\ref{S:matching}).

\subsection*{Notation}

We denote for $r\in \R$, 
\[
\langle r\rangle = \sqrt{1+|r|^2}.
\]
Let $d\geq 1$. Let $\bar p>p_*>1$ and $p\in [p_*,\bar p]$.
All constants $C>0$ are independent of~$p$, $\sigma$, $b$, $\rho$, $\gamma$ and $\theta$ but may depend on $d$ and $\bar p$.
For any $A\subset \R$, and for any functions $g:A\to \C$, $f:A\to [0,\infty)$, the notation $g(r)=O(f(r))$ means that there exists a constant $C>0$, independent of $p$, $\sigma$, $b$, 
$\rho$, $\gamma$, $\theta$ and $A$, such that it holds $|g(r)|\leq C f(r)$ for any $r\in A$.

\subsection*{Acknowledgements} Y.B. is partially supported by the ERC-2014-CoG 646650 SingWave. P.R. is supported by the ERC-2014-CoG 646650 SingWave. 
Y.M. would like to thank the DPMMS, University of Cambridge for its hospitality.
P.R. would like to thank the Universit\'e de la C\^ote d'Azur where part of this work was done for its kind hospitality. The authors thank E. Lombardi (Toulouse) and T. Cazenave (Paris 6) for enlightening discussions. The authors are also grateful to S. Aryan (\'Ecole polytechnique) for his careful reading of the manuscript.

\section{Solutions of the nonlinear equation on $I$}
In this section, we construct a family of solutions $U$ of \eqref{eq:U} on the interval $I=[b^{-2},\infty)$,
with admissible behavior at $\infty$, see Proposition \ref{pr:UonI} and Remark~\ref{rk:h1dot}.
The first step is to investigate the asymptotic behavior of solutions of the corresponding linear problem
(also neglecting for now the term $\frac{(d-1)(d-3)}{4r^2}$).

\subsection{Construction of approximate linear profiles at infinity}
We construct approximate solutions of the equation
\begin{equation}\label{eq:linv}
V''+ \left( \frac{b^2r^2}4-1 -\ii b\sigma\right) V=0
\end{equation}
of the form
\begin{equation}\label{def:V}
V(r)=\exp\left( \frac{\theta(s)}b \right),\quad s = br.
\end{equation}
where $\theta$ is a function to be chosen.
Using the form \eqref{def:V} into the equation \eqref{eq:linv} gives
\begin{equation}\label{eq:theta}
 (\theta')^2 + b \theta'' + \frac{s^2}{4} -1 - \ii b\sigma=0.
\end{equation}
 Set
\begin{equation*}
\theta = \ii \theta_0 + b \theta_1.
\end{equation*}
Identifying terms of order $b^0$, we find 
\[
(\theta_0')^2(s) = \frac 14 (s^2-4),\quad s\geq 2.
\]

(i) We choose $\theta_0'(s) = \frac 12 \sqrt{s^2-4}$, which can be integrated explicitly as
\begin{equation*}
\theta_0(s) = \frac s4\sqrt{s^2-4} - \ln(\sqrt{s^2-4}+s),\quad s\geq 2.
\end{equation*}
Inserting the explicit form of $\theta_0$ in \eqref{eq:theta}, and identifying terms of order $b^1$, we find the equation
of $\theta_1$
\[
\theta_1' = - \frac 1 2 \frac{\theta_0''}{\theta_0'} + \frac{ \sigma}{\sqrt{s^2-4}},\quad s>2,
\]
which can also be integrated explicitly. We set
\begin{equation*}
\theta_1^+(s) = - \frac 14\ln(s^2-4) + \sigma \ln(\sqrt{s^2-4}+s),\quad s> 2.
\end{equation*}
Setting $\theta^+=\ii\theta_0+b\theta_1^+$, \eqref{eq:theta} is not exactly solved but we have obtained instead
\begin{equation*}
 \left[(\theta^+)'\right]^2 + b (\theta^+)'' + \frac{s^2}{4} -1 - \ii b\sigma= b^2 f^+,
\end{equation*}
where
\begin{equation*}
f^+(s)= \left[(\theta_1^+)'\right]^2+ (\theta_1^+)'' .
\end{equation*}
The function $f^+$ depends on $\sigma$,
but on $[4,\infty)$, it holds $f^+(s)=O(s^{-2})$ and more generally, $(f^+)^{(k)}=O(s^{-2-k})$, for any $k\geq 0$.

(ii) Setting $\theta^-=-\ii \theta_0+b\theta_1^-$, 
where
\begin{equation*}
\theta_1^-(s) = -\frac14\ln(s^2-4)- \sigma \ln(\sqrt{s^2-4}+s),\quad s> 2,
\end{equation*}
we obtain
\begin{equation*}
 \left[(\theta^-)'\right]^2 + b (\theta^-)'' + \frac{s^2}{4} -1 - \ii b\sigma
= b^2 f^-, \quad f^-=[(\theta_1^-)']^2+(\theta_1^-)'',
\end{equation*}
and so, on $[4,\infty)$, $(f^-)^{(k)}=O(s^{-2-k})$, for any $k\geq 0$.
\begin{lemma}\label{le:VpVm}
Let $\sigma>0$ be small and let $b$ be as in \eqref{on:b}.
There exist smooth functions $V^\pm:(\frac 2b,\infty)\to \C$ satisfying
\begin{enumerate}
\item\emph{Equation of $V^\pm$:}
\begin{equation}\label{eq:Vfinal}
(V^\pm)''+ \left( \frac{b^2r^2}4-1 -\ii b\sigma\right) V^\pm = b^2 f^\pm(br) V^\pm,\quad r>\frac 2b,
\end{equation}
where on $[\frac 4b,\infty)$, 
\begin{equation*}
(f^\pm)^{(k)}=O(s^{-2-k}), \quad \mbox{for any $k\geq 0$.}
\end{equation*}
\item \emph{Asymptotics of $V^\pm$:} on the interval $I$, it holds
\begin{align}
V^\pm & = r^{-\frac 12 \pm\sigma} \exp\left(\pm\ii b\frac{r^2}4\right)\exp\left(\mp\ii \frac{\ln r}{b}\right) \left(1+O(b^{-3}r^{-2})\right),\label{on:V}\\
(V^\pm)' & = \pm\frac \ii 2 b r^{\frac 12 \pm\sigma} \exp\left(\pm\ii b\frac{r^2}4\right)\exp\left(\mp\ii \frac{\ln r}{b}\right)
\left(1+O(b^{-3}r^{-2})\right),\label{on:dV}\\
|V^\pm | & = r^{-\frac 12\pm\sigma} \left(1+O(b^{-2}r^{-2})\right),\label{on:mV}
\end{align}
and
\begin{equation}\label{on:VdV}
(V^\pm)' \mp \ii \frac{br}2V^\pm =O(b^{-1}r^{-1}|V^\pm |)=O(b^{-1}r^{-\frac 32\pm\sigma}).
\end{equation}
\item \emph{Wronskian of $V^\pm$:}
\begin{equation}\label{W:Vpm}
\mathcal W(V^+,V^-)=V^+(V^-)'-(V^+)'V^-=
-\ii b - \frac {2 b^2\sigma}{b^2r^2-4}.
\end{equation}
\end{enumerate}
Moreover, the maps $(\sigma,b,r)\mapsto (V^\pm[\sigma,b](r),V^\pm[\sigma,b]'(r))$ are continuous.
\end{lemma}
\begin{remark}
The Wronskian $\mathcal W(V^+,V^-)$ is not constant since $V^+$ and $V^-$ do not satisfy the same linear equation; however, it does not vanish and it has a non zero limit as $r\to \infty$.
\end{remark}
\begin{remark}
The point of constructing explicitly approximate solutions $V^\pm$ of \eqref{eq:linv} is to be able to track very precisely their asymptotic behavior for $r>\frac 2b$, as well as the size of the error term, in terms of  the small parameters $b$ and $\sigma$. The error term will be easily absorbed by the fixed-point procedure needed anyway for the full nonlinear problem \eqref{eq:U}.
\end{remark}
\begin{proof}
Proof of (i)-(ii).
We set
\[
V_0(r)=\exp\left(\ii \frac{\theta_0(br)}b\right),\quad
V_1^+(r)=\exp\left(\theta_1^+(br)\right),\quad V_1^-(r)=\exp\left(\theta_1^-(br)\right).
\]
We study the asymptotics of $V_0$ and $V_1^\pm$.
First, for $V_0$, by Taylor expansion, we observe
\begin{align*}
 \frac{\theta_0(br)}{b}
&= \frac r4\sqrt{b^2r^2-4} - \frac 1b\ln(\sqrt{b^2r^2-4}+br)\\
&= b\frac{r^2}4 
-\frac {\ln r}b-\frac 1{2b}-\frac {\ln(2b)}b+O(b^{-3}r^{-2}).
\end{align*}
Thus,
\begin{equation*}
V_0(r)=\exp\left(\ii b\frac{r^2}4\right)
\exp\left(-\ii\frac {\ln r}{b} \right)\exp\left(-\ii\frac {1+2\ln(2b)}{2b}\right)\left(1+O(b^{-3}r^{-2})\right).
\end{equation*}
Moreover, since $V_0'(r)=\ii\theta_0'(br) V_0(r)=\ii\frac 12 (b^2r^2-4)^{\frac 12}V_0(r)$, we also have
\begin{equation*}
V_0'(r)=\ii\frac{br}2\exp\left(\ii b\frac{r^2}4\right)
\exp\left(-\ii\frac {\ln r}{b} \right)\exp\left(-\ii\frac {1+2\ln(2b)}{2b}\right)\left(1+O(b^{-3}r^{-2})\right).
\end{equation*}
Second, we have
\begin{align*}
V_1^+(r)
&=(b^2r^2-4)^{-\frac 14} \left(\sqrt{b^2r^2-4}+br\right)^\sigma
=(br)^{-\frac 12} (2br)^\sigma \left(1+O(b^{-2}r^{-2})\right),\\
V_1^-(r)&=(b^2r^2-4)^{-\frac 14} \left(\sqrt{b^2r^2-4}+br\right)^{-\sigma}
=(br)^{-\frac 12} (2br)^{-\sigma} \left(1+O(b^{-2}r^{-2})\right)
\end{align*}
Moreover, since $(V_1^+)'(r)= b(\theta_1^+)'(br)V_1^+(r)$, it holds $(V_1^+)'=O(r^{-1} V_1^+)$.
Similarly, $(V_1^-)'=O(r^{-1} V_1^-)$.

We set
\begin{align*}
V^+ &= \exp\left(\ii\frac {1+2\ln(2b)}{2b}\right) b^{\frac 12}(2b)^{-\sigma} V_0V_1^+,
\\
V^- &= \exp\left(-\ii\frac {1+2\ln(2b)}{2b}\right) b^{\frac 12}(2b)^{\sigma} \overline V_0V_1^-.
\end{align*}
(Observe that $V^-$ is also obtained by taking the complex conjugate of $V^+$ and changing $\sigma$ into $-\sigma$.)
The continuity property is clear by the explicit expression of $V^\pm$.
By the computations preceding the lemma, the functions $V^\pm$ satisfy~\eqref{eq:Vfinal}.
Moreover, the above asymptotic computations imply the asymptotics \eqref{on:V}-\eqref{on:dV}.
Since $|V_0|=1$, we also have
\[
|V^\pm(r)|=b^{\frac 12}(2b)^{\mp\sigma} |V_1^\pm(r)|
=r^{-\frac 12\pm\sigma} \left(1+O(b^{-2}r^{-2})\right),
\]
which is \eqref{on:mV}. 
Last, we observe that
\begin{align*}
(V^+)'-\ii \frac{br}2V^+ & = V^+\left( \ii \theta_0'(br)-\ii \frac {br}2+b(\theta_1^+)'(br)\right)\\
& =V^+\left( \ii \frac{br}2\left[(1-4b^{-2}r^{-2})^{\frac 12}-1\right]+b(\theta_1^+)'(br)\right)
 = O(b^{-1}r^{-1}|V^+|).
\end{align*}
Together with a similar computation for $V^-$, this proves the estimates in \eqref{on:VdV}.

Proof of (iii). We recall that
$V_0\overline V_0=|V_0|^2=1$, $V_1^+V_1^-=(b^2r^2-4)^{-\frac 12}$, and 
$V_0'=\ii \theta_0'(br)V_0$, $\overline V_0'=-\ii \theta_0'(br)\overline V_0$,
$(V_1^\pm)'=b(\theta_1^\pm)'(br)V_1^\pm$. Thus
\begin{align*}
\mathcal W(V^+,V^-)&=
b\left[V_0V_1^+\left(\overline V_0'V_1^-+\overline V_0(V_1^-)'\right)
-\left(V_0'V_1^++V_0(V_1^+)'\right)(\overline V_0V_1^-)\right]\\
&=-\ii b +b^2(b^2r^2-4)^{-\frac 12} \left( (\theta_1^-)'(br)-(\theta_1^+)'(br)\right)
=-\ii b- \frac {2 b^2\sigma}{b^2r^2-4},
\end{align*}
which completes the proof of the lemma.
\end{proof}

\subsection{Construction of a family of solutions of \eqref{eq:U} on $I$}
Using the method of variation of constants and the approximate solutions of the linear problem
given in Lemma~\ref{le:VpVm}, we construct by a fixed-point procedure a family of solutions of \eqref{eq:U}
whose asymptotic behavior follows the one of the function $V^+$.
The additional free parameter $\rho$ introduced in the following result corresponds to the amplitude of each solution of the family.
\begin{proposition}\label{pr:UonI}
For $\sigma>0$ small enough and
for any $b$, $\rho$ satisfying \eqref{on:b}, \eqref{on:rho}, there exists a $\mathcal C^2$ solution $U$ of \eqref{eq:U} on $I$ satisfying
\begin{align}
U & =\rho r^{-\frac 12 +\sigma} \exp\left(\ii b\frac{r^2}4\right)\exp\left(-\ii \frac{\ln r}{b}\right) \left(1+O(b^{-3}r^{-2})\right),\label{on:U}\\
U' & = \ii \frac{b}2 \rho r^{\frac 12 +\sigma} \exp\left(\ii b\frac{r^2}4\right)\exp\left(-\ii \frac{\ln r}{b}\right)
\left(1+O(b^{-3}r^{-2})\right),\label{on:dU}
\end{align}
and
\begin{equation}\label{on:UdU}
 U'-\ii \frac{br}2U=O(b^{-1} r^{-1} |U|)= O(\rho b^{-1} r^{-\frac 32+\sigma}).
\end{equation}
Moreover, the map $(\sigma,b,\rho)\mapsto (U[\sigma,b,\rho](r_I),U[\sigma,b,\rho]'(r_I))$
is continuous.
\end{proposition}
\begin{remark}\label{rk:h1dot}
We check that with such asymptotics, the function $\Psi$ 
defined on $I$ by
\begin{equation*}
\Psi(r)=r^{-\frac{d-1}2} U(r) \exp \left(-\ii \frac{br^2}4 \right)
\end{equation*}
does not belong to $L^2(|x|>r_I)$, but belongs to $\dot H^1(|x|>r_I)$.
Indeed, it holds
\begin{equation*}
\Psi'(r)=\left(-\frac {d-1}2 r^{-1}U(r)+U'(r)-\ii \frac {br}2 U(r)\right) r^{-\frac {d-1}2} \exp\left(-\ii \frac{br^2}4\right),
\end{equation*}
and so by \eqref{on:UdU},
$\Psi'=O (\rho b^{-1} r^{-\frac {d-1}2-\frac 32+\sigma})$. Thus,
\[
\int_{r_I}^\infty |\Psi'(r)|^2 r^{d-1} dr
\leq C\rho ^2 b^{-2}\int_{r_I}^\infty r^{-3+2\sigma} dr\leq C\rho ^2 b^{-2} r_I^{-2+2\sigma}\leq C\rho ^2 b^{2-4\sigma}.
\]
\end{remark}
\begin{proof}
We apply the method of variation of constants, looking for a solution $U$ of \eqref{eq:U} under the form
\begin{equation}\label{def:U}\left\{\begin{aligned}
U&=\lambda^+ V^+ + \lambda^- V^-\\
U'&=\lambda^+ (V^+)' + \lambda^- (V^-)'.
\end{aligned}\right.\end{equation}
Recall from \eqref{W:Vpm} that $w(r)=\mathcal W(V^+,V^-)=-\ii b - \frac {2 b^2\sigma}{b^2r^2-4}$ never vanishes.
By standard computations and \eqref{eq:Vfinal}, we obtain the system
\begin{equation}\label{syst:U}\left\{\begin{aligned}
(\lambda^+)' V^+ + (\lambda^-)' V^- &= 0\\
(\lambda^+)' (V^+)' + (\lambda^-)' (V^-)'&=F(r,\lambda^+ V^+ + \lambda^- V^-)
\end{aligned}\right.\end{equation}
where
\begin{align*}
F(r,U)&=-\lambda^+b^2f^+(br)V^+- \lambda^-b^2f^-(br)V^-
\\&\quad +\frac{(d-1)(d-3)}{4r^2}U-r^{-\frac 12(d-1)(p-1)}|U|^{p-1}U.
\end{align*}
The system \eqref{syst:U} writes equivalently
\begin{equation}\label{syst:la}\left\{\begin{aligned}
(\lambda^+)' &= - \frac{V^-}{w} F(r,\lambda^+ V^+ + \lambda^- V^-)\\
(\lambda^-)' &= \frac{V^+}{w} F(r,\lambda^+ V^+ + \lambda^- V^-).
\end{aligned}\right.\end{equation}
For $\rho \in (0,1]$, we set
\begin{align*}
\Gamma^+[\lambda^+,\lambda^-]&=\rho+\Gamma^+_1 [\lambda^+,\lambda^-]+\Gamma^+_2 [\lambda^+,\lambda^-]+\Gamma^+_3 [\lambda^+,\lambda^-]+\Gamma^+_4 [\lambda^+,\lambda^-],\\
\Gamma^-[\lambda^+,\lambda^-]&=\Gamma^-_1 [\lambda^+,\lambda^-]+\Gamma^-_2 [\lambda^+,\lambda^-]+\Gamma^-_3 [\lambda^+,\lambda^-]
+\Gamma^-_4 [\lambda^+,\lambda^-],
\end{align*}
where
\begin{align*}
&\Gamma^+_1 [\lambda^+,\lambda^-]=-\int_r^\infty \frac{V^-(r')}{w(r')} \lambda^+(r')b^2f^+(br')V^+(r')dr',\\
&\Gamma^+_2 [\lambda^+,\lambda^-]=-\int_r^\infty \frac{V^-(r')}{w(r')} \lambda^-(r')b^2f^-(br')V^-(r')dr',\\
&\Gamma^+_3 [\lambda^+,\lambda^-]=\int_r^\infty \frac{V^-(r')}{w(r')} \frac{(d-1)(d-3)}{4(r')^2}(\lambda^+ V^+ + \lambda^- V^-)(r')dr',\\
&\Gamma^+_4 [\lambda^+,\lambda^-]\\
&\quad =-\int_r^\infty \frac{V^-(r')}{w(r')} (r')^{-\frac 12(d-1)(p-1)}|\lambda^+ V^+ + \lambda^- V^-|^{p-1}(\lambda^+ V^+ + \lambda^- V^-)(r')dr',
\end{align*}
and
\begin{align*}
&\Gamma^-_1 [\lambda^+,\lambda^-]=\int_{r}^\infty \frac{V^+(r')}{w(r')} \lambda^+(r')b^2f^+(br')V^+(r')dr',\\
&\Gamma^-_2 [\lambda^+,\lambda^-]=\int_{r}^\infty \frac{V^+(r')}{w(r')}\lambda^-(r')b^2f^-(br')V^-(r')dr',\\
&\Gamma^-_3 [\lambda^+,\lambda^-]=-\int_{r}^\infty \frac{V^+(r')}{w(r')} \frac{(d-1)(d-3)}{4(r')^2}(\lambda^+ V^+ + \lambda^- V^-)(r')dr',\\
&\Gamma^-_4 [\lambda^+,\lambda^-]\\
&\quad=\int_{r}^\infty \frac{V^+(r')}{w(r')} (r')^{-\frac 12(d-1)(p-1)}|\lambda^+ V^+ + \lambda^- V^-|^{p-1}(\lambda^+ V^+ + \lambda^- V^-)(r')dr'.
\end{align*}
We work in the following complete metric space
\begin{align*}
E_I&=\left\{(\lambda^+,\lambda^-): I\to \C^2 \mbox{ is continuous and satisfies 
$\|(\lambda^+,\lambda^-)\|_I\leq 2\rho $}\right\}
\end{align*}
equipped with the distance associated to the norm
\[
\|(\lambda^+,\lambda^-)\|_I
=\sup_{I} \big(|\lambda^+|+|\lambda^- |\big).
\]
Using \eqref{on:mV}, we have the bounds, on $I$,
\begin{align*}
|\Gamma_1^+[\lambda^+,\lambda^-]|+|\Gamma_2^+[\lambda^+,\lambda^-]|+|\Gamma_3^+[\lambda^+,\lambda^-]|
&\leq C\|(\lambda^+,\lambda^-)\|_I b^{-1}\int_{r}^{\infty} (r')^{-3} dr'\\
& \leq C \|(\lambda^+,\lambda^-)\|_I b^{-1} r^{-2}.
\end{align*}
For the term $\Gamma_4^+$, we first observe using \eqref{on:mV} that
\[
|V^-|(|\lambda^+||V^+|+|\lambda^-||V^-|)^p
\leq C \|(\lambda^+,\lambda^-)\|_I^p r^{-\frac{p+1}2+\sigma(p-1)}.
\]
Thus, using also $p\geq p_*=1+\frac 4d$, one has
\begin{align*}
|\Gamma_4^+[\lambda^+,\lambda^-]|
&\leq C\|(\lambda^+,\lambda^-)\|_I^p b^{-1} \int_{r}^\infty
(r')^{-\frac 12(d-1)(p-1)}(r')^{-\frac{p+1}2+\sigma(p-1)}dr'\\
&\leq C \|(\lambda^+,\lambda^-)\|_I^p b^{-1} \int_{r}^\infty (r')^{-3+\sigma(p-1)} dr'\\
&\leq C\|(\lambda^+,\lambda^-)\|_{I} b^{-1} r^{-2+\sigma(p-1)}.
\end{align*}
Moreover, using the inequality
\[
\left||U|^{p-1}U-|\tilde U^{p-1}|\tilde U\right|\leq C (|U|^{p-1}+|\tilde U|^{p-1}) |U-\tilde U|,
\]
for any $(\lambda^+,\lambda^-)$, $(\tilde\lambda^+,\tilde\lambda^-)\in E_I$, and the above estimates,
we have
\begin{equation}\label{contrac}
|\Gamma_4^+[\lambda^+,\lambda^-]-\Gamma_4^+[\tilde\lambda^+,\tilde\lambda^-]|
\leq C\|(\lambda^+,\lambda^-)-(\tilde \lambda^+,\tilde \lambda^-)\|_I b^{-1}r^{-2}.
\end{equation}
We also have
\begin{align*}
&|\Gamma_1^-[\lambda^+,\lambda^-]|+|\Gamma_2^-[\lambda^+,\lambda^-]|+|\Gamma_3^-[\lambda^+,\lambda^-]|\\
&\qquad \leq C\|(\lambda^+,\lambda^-)\|_I b^{-1}\int_{r}^{\infty} (r')^{-3+\sigma(p+1)} dr'\\
&\qquad \leq C\|(\lambda^+,\lambda^-)\|_Ib^{-1} r^{-2+\sigma(p+1)};
\end{align*}
Moreover, as before
\begin{align*}
|\Gamma_4^-[\lambda^+,\lambda^-]|
&\leq C\|(\lambda^+,\lambda^-)\|_I^p b^{-1} \int_{r}^\infty
(r')^{-\frac 12(d-1)(p-1)}(r')^{-\frac{p+1}2+\sigma(p+1)}dr'\\
&\leq C \|(\lambda^+,\lambda^-)\|_I^p b^{-1} \int_{r}^\infty (r')^{-3+\sigma(p+1)} dr'\\
&\leq C\|(\lambda^+,\lambda^-)\|_I b^{-1} r^{-2+\sigma(p+1)}.
\end{align*}
and
\begin{equation}\label{contracm}
|\Gamma_4^-[\lambda^+,\lambda^-]-\Gamma_4^-[\tilde\lambda^+,\tilde\lambda^-]|
\leq C\|(\lambda^+,\lambda^-)-(\tilde \lambda^+,\tilde \lambda^-)\|_I b^{-1} r^{-2+\sigma(p+1)} .
\end{equation}
Recall that we work here for $r\in I$, \emph{i.e.} $r\geq b^{-2}$.
It follows from these estimates that for $b$ and $\sigma$ small enough, $(\Gamma^+,\Gamma^-)$ maps $E_I$ into $E_I$.
Moreover, by linearity of $\Gamma_1^\pm$, $\Gamma_2^\pm$ and $\Gamma_3^\pm$ in $(\lambda^+,\lambda^-)$, and by the estimates \eqref{contrac} and \eqref{contracm},
it follows that $(\Gamma^+,\Gamma^-)$ is a contraction on $E_I$ (for $b$ small enough).

The application $(\Gamma^+,\Gamma^-):E_I\to E_I$ thus admits a unique fixed point $(\lambda^+,\lambda^-)$ on $E$.
Moreover, $(\lambda^+,\lambda^-)$ are of class $\mathcal C^1$ on $I$, satisfy \eqref{syst:la} (equivalently \eqref{syst:U})
 and
\begin{equation*}
|\lambda^+(r)-\rho|\leq C\rho b^{-1}r^{-2},\quad|\lambda^-(r)|\leq C\rho b^{-1} r^{-2+\sigma(p+1)}.\end{equation*}
Inserting the above estimates into \eqref{def:U} yields
\[
U=\rho V^+(1+O( b^{-1}r^{-2})),\quad
U'=\rho (V^+)'(1+O(b^{-1}r^{-2})).
\]
Using \eqref{on:V} and \eqref{on:dV}, we find \eqref{on:U} and \eqref{on:dU}.
Using \eqref{on:VdV}, we also find \eqref{on:UdU}.

The continuity of the map $(\sigma,b,\rho,r)\mapsto (U[\sigma,b,\rho](r),U[\sigma,b,\rho]'(r))$ follows from the continuity of
the functions $F$, $V^\pm$ and their derivatives, and the application of the Banach Fixed-Point theorem with parameters.
\end{proof}

\section{Solution of the nonlinear equation on $J\cup I$}
In this section, we prove that the solutions of \eqref{eq:U} constructed on $I=[b^{-2},\infty)$ in Proposition~\ref{pr:UonI}
actually extend to the interval $J\cup I= [b^{-\frac 12},\infty)$, and we describe precisely their behavior at $r=r_K=b^{-\frac 12}$. A classical difficulty is that the equation \eqref{eq:U} has a turning point around $r_J=2b^{-1}$.
To deal with this, we follow the approach described in \cite{SulemSulem2,SulemSulem}, \cite{Fedoryuk} and implemented in a similar context in \cite{Godet}. It consists in changing variables to reduce to the use of classical Airy functions, independent solutions of $y''(s)=s\, y(s)$.
An alternative approach is to use a richer class of special functions (hypergeometric functions, solutions of the Kummer equations), as in \cite{KL}, \cite{PS} and \cite{RK}.
As in the previous section, we have preferred to reduce to the simplest possible setting to track precisely the dependence on the small parameters.

\subsection{Preliminary on Airy functions}
Recall that the classical Airy function $\Ai:\C\to \C$ is defined by
\[
\Ai(z)=\frac 1{2\pi} \int_{\R+\ii \eta} \exp\left(\ii\left( z\xi + \frac{\xi^3}3\right)\right) d\xi,\quad \eta>0,
\]
expression which is independent of $\eta>0$.
Define
\[
\BB(z)=2\pi \ee^{\ii \frac \pi 6} \Ai\left(\ee^{\frac{2\ii\pi}{3}}z\right).
\]
We recall some properties of $\Ai$ and $\BB$ seen as functions of a real variable
(see \emph{e.g.} \cite{Fedoryuk} and \cite{lernerbook}).
\begin{lemma}\label{le:Airy}
\begin{enumerate}
\item \emph{The Airy function on $\R$.}
The function $\Ai$ on $\R$ is real-valued. Moreover, it satisfies the following asymptotics:
\begin{align*}
\Ai(s)&= \frac{1}{\sqrt{\pi}} |s|^{-\frac 14}\Re\left\{ \ee^{\ii \frac{\pi}4} \exp\left(-\frac 23\ii|s|^{\frac 32}\right)\right\} \left(1+O(|s|^{-\frac 32})\right) \quad\mbox{on $(-\infty,-1]$},\\
\Ai'(s)&=- \frac{1}{\sqrt{\pi}} |s|^{\frac 14}\Im\left\{ \ee^{\ii \frac{\pi}4} \exp\left(-\frac 23\ii|s|^{\frac 32}\right)\right\} \left(1+O(|s|^{-\frac 32})\right) \quad\mbox{on $(-\infty,-1]$},
\end{align*}
and
\begin{align*}
\Ai(s)&=\frac 1{2\sqrt{\pi}}s^{-\frac 14} \exp\left(-\frac 23 s^{\frac 32}\right) \left( 1+ O(s^{-\frac 32})\right) \quad \mbox{on $[1,\infty)$},\\
\Ai'(s)&=-\frac 1{2\sqrt{\pi}} s^{\frac 14}\exp\left(-\frac 23 s^{\frac 32}\right) \left( 1+ O(s^{-\frac 32})\right) \quad \mbox{on $[1,\infty)$}.
\end{align*}
\item \emph{The complex-valued Airy function on $\R$.}
The function $\BB$ on $\R$ does not vanish. Moreover, it satisfies the following asymptotics:
\begin{align*}
\BB(s)&= {\sqrt{\pi}} |s|^{-\frac 14} \ee^{\ii \frac{\pi}4} \exp\left(\frac 23\ii|s|^{\frac 32}\right) \left(1+O(|s|^{-\frac 32})\right) \quad\mbox{on $(-\infty,-1]$},\\
\BB'(s)&=-\ii{\sqrt{\pi}} |s|^{\frac 14} \ee^{\ii \frac{\pi}4} \exp\left(\frac 23\ii|s|^{\frac 32}\right) \left(1+O(|s|^{-\frac 32})\right) \quad\mbox{on $(-\infty,-1]$},
\end{align*}
and
\begin{align*}
\BB(s)&=\sqrt{\pi}s^{-\frac 14} \exp\left(\frac 23 s^{\frac 32}\right) \left( 1+ O(s^{-\frac 32})\right) \quad \mbox{on $[1,\infty)$},\\
\BB'(s)&=\sqrt{\pi} s^{\frac 14}\exp\left(\frac 23 s^{\frac 32}\right) \left( 1+ O(s^{-\frac 32})\right) \quad \mbox{on $[1,\infty)$}.
\end{align*}
\item \emph{Relations between $\Ai$ and $\BB$.} The functions $\Ai$ and $\BB$ are independent solutions of
the linear equation $y''=sy$ on $\R$. Moreover,
\begin{equation}\label{wrAiB}
\mathcal W(\Ai,\BB)=\Ai\BB'-\Ai'\BB=1,
\end{equation}
and
\begin{equation}\label{ImB}
\Im \BB = \pi \Ai.
\end{equation}
\end{enumerate}
\end{lemma}
\begin{remark}
In particular, setting
\begin{equation*}
\omega(s) = \exp\left(\frac 23 s_+^{\frac 32}\right)\quad \mbox{where $s_+=\max (0,s)$},
\end{equation*}
there exists $C>1$ such that, for all $s\in\R$,
\begin{align}
|\Ai(s)|&\leq C \langle s\rangle^{-\frac 14} [\omega(s)]^{-1},\label{airy1}\\
\frac 1C\langle s\rangle^{-\frac 14} \omega(s)\leq |\BB(s)|&\leq C \langle s\rangle^{-\frac 14} \omega(s),\label{airy2}\\
|\Ai(s)\BB(s)|&\leq C\langle s\rangle^{-\frac 12}.\label{airy3}
\end{align}
\end{remark}

\subsection{Construction of a family of solutions on $J$}

To construct solutions of the nonlinear equation \eqref{eq:U} on $J$, 
we adapt the strategy of the proof of Lemma~1 in~\cite{Godet} (see also \cite{Fedoryuk}, Chapter~4).
We perform several changes of variables to deal with the (approximate) turning point $r_J=2b^{-1}$ of the equation~\eqref{eq:U}, using the classical Airy functions.
We reduce to the method of variation of constants using $\Ai$ and $\BB$, and to a fixed-point procedure to deal with the nonlinear term and some linear error terms.
Once general solutions of \eqref{eq:U} are constructed on $J$, we use them to extend the solutions constructed in Proposition~\ref{pr:UonI} on~$J\cup I$. Last, we relate the special asymptotic behavior of these solutions to the special form \eqref{Pext1}-\eqref{Pext4} at the point $r=r_K$, after a change of phase and tracking precisely the (exponentially small) dependence in the parameter $b$.

\begin{proposition}\label{pr:Uext}
For $\sigma>0$ small enough and for any $b$, $\rho$ satisfying \eqref{on:b}, \eqref{on:rho},
the solution $U$ of \eqref{eq:U} on $I$ constructed in Proposition~\ref{pr:UonI}
 extends to a solution of \eqref{eq:U} on $J\cup I$. Moreover,
there exists a real $\theta_\ext\in [0,2\pi)$ such that the function $P_\ext$ defined by
\begin{equation}\label{def:Pext}
P_\ext(r) = \ee^{\ii \theta_\ext} r^{-\frac{d-1}2} U(r), \quad \mbox{for any $r\in J\cup I$,}
\end{equation}
is a solution of~\eqref{eq:Pb} on $J\cup I$ and satisfies
\begin{align}
\Re(P_\ext(r_K))&=\frac{\rho b^{\frac 12 +\frac{d-1}4}}{\sqrt{2}}
\exp\left(\frac{\pi}{2b}\right) \exp\left(-\frac1{\sqrt{b}}\right) (1+O(b^{\frac 14})),\label{Pext1}\\
\Re(P_\ext'(r_K))&= -\frac{\rho b^{\frac 12 +\frac{d-1}4}}{\sqrt{2}}\exp\left(\frac{\pi}{2b}\right) \exp\left(-\frac1{\sqrt{b}}\right) (1+O(b^{\frac 14})),\label{Pext2}\\
\Im(P_\ext(r_K))&=\frac{\rho b^{\frac 12 +\frac{d-1}4}}{2\sqrt{2}} 
\exp\left(-\frac{\pi}{2b}\right) \exp\left(\frac1{\sqrt{b}}\right) (1+O(b^{\frac 14}))\label{Pext3} ,\\
\Im(P_\ext'(r_K))&= \frac{\rho b^{\frac 12 +\frac{d-1}4}}{2\sqrt{2}}\exp\left(-\frac{\pi}{2b}\right) \exp\left(\frac1{\sqrt{b}}\right) (1+O(b^{\frac 14})).\label{Pext4}
\end{align}
Moreover, the map $(\sigma,b,\rho)\mapsto (P_\ext[\sigma,b,\rho](r_K),P_\ext[\sigma,b,\rho]'(r_K))$ is continuous.
\end{proposition}
\begin{proof} 
We proceed in two steps as described above. First, we give a general construction argument on $J$, for any given
small data at $r_I$. Second, we extend the solutions given by Proposition~\ref{pr:UonI} and 
describe precisely their  behavior at $r_K$.

\textbf{Step 1.} Construction of a family of solutions on $J$.
Let
\begin{equation*}
\Sigma_b = \exp\left(-\frac \pi {2b}+\frac 1{2\sqrt{b}}\right).
\end{equation*}
Let $u_I$, $\tilde u_I\in \C$ be such that
\begin{equation}\label{on:uIduI}
|u_I|\leq b \Sigma_b ,\quad |\tilde u_I|\leq \Sigma_b.
\end{equation}
In this step, we construct a solution $U_J$ of \eqref{eq:U} on $J$ satisfying
\begin{equation}\label{ini:U_J}
U_J(r_I)=u_I,\quad 
U_J'(r_I)=\tilde u_I.
\end{equation}
For the equation~\eqref{eq:U} with $d\neq 1$ and $d\neq 3$, the turning point is not exactly at $r_J$, but we still use the following change of variable
\begin{align*}
&\tau = 2-br, \quad br=2-\tau, \quad \tau\leq 2,
\\
&W(\tau)=W(2-br)=U(r),\quad U''(r)=b^2 W''(\tau).
\end{align*}
Setting
\[
h(\tau)=\frac 12 \sqrt{4-\tau},
\]
the equation of $W$ writes, for $\tau>2$,
\[
b^2 W''-\tau h^2(\tau) W - b^2 \frac{(d-1)(d-3)}{4(2-\tau)^2} W
-\ii b\sigma W+ \left(\frac b{2-\tau}\right)^{\frac 12 (d-1)(p-1)} |W|^{p-1} W=0.
\]
Define the function
\begin{equation}\label{def:zeta}
\zeta(\tau)=\begin{cases} \left(\frac 32\int_0^\tau \sqrt{\tau'} h(\tau') d\tau'\right)^{\frac 23}, \quad \tau\in [0,2],\\
- \left(\frac 32\int_\tau^0 \sqrt{|\tau'|} h(\tau') d\tau'\right)^{\frac 23}, \quad \tau\in (-\infty,0].
\end{cases}
\end{equation}
The function $\zeta$ can be determined explicitly but we will not need its expression.
Note that 
\begin{equation}\label{on:rho0}
\zeta_0=\left(\frac {3\pi}4\right)^{\frac 23},\quad \zeta(2)=\zeta_0,\quad \zeta'(2)=\left(\frac {3\pi}4\right)^{-\frac 13}.
\end{equation}
Moreover,
\begin{align*}
\zeta(\tau)&=\zeta(2) - (2-\tau) \zeta'(2) +O\left((2-\tau)^2\right) \quad \mbox{on $[1,2]$,}\\
\zeta(\tau)&= \tau (1+O(|\tau|)) \quad \mbox{on $[-1,1]$}\\
\zeta(\tau)&= - \frac{3^{\frac 23}}{4} |\tau|^{\frac 43} \left(1+O(|\tau|^{-1})\right) \quad \mbox{on $(-\infty,-1]$,}
\end{align*}
and
\begin{align}
\zeta'(\tau)&=\sqrt{|\tau|}h(\tau)|\zeta(\tau)|^{-\frac 12}=3^{-\frac 13}|\tau|^{\frac 13}(1+O(|\tau|^{-1}))
\quad \mbox{on $(-\infty,-1]$,}\label{on:rho4}\\
\zeta''(\tau)&=-\frac 12 |\tau|^{-\frac 12}(4-\tau)^{-\frac 12}(2-\tau)|\zeta(\tau)|^{-\frac 12}
+\frac 18 |\tau|(4-\tau)|\zeta(\tau)|^{-2}\nonumber\\
&=-3^{-\frac 43}|\tau|^{-\frac 23}(1+O(|\tau|^{-1}))\quad \mbox{on $(-\infty,-1]$.}\label{on:rho5}
\end{align}
We define the function $X$ of the variable $\zeta$ by
\[
X(\zeta(\tau))=W(\tau) \sqrt{\zeta'(\tau)}.
\]
Inserting $W(\tau)=\frac{X(\zeta(\tau))}{\sqrt{\zeta'(\tau)}}$ in the equation of $W$, we obtain in the variable $\tau\in (-\infty,2)$,
the following equation
\begin{align*}
b^2X''(\zeta)
&=\zeta X(\zeta) + \frac{b^2}2\left(\frac{\zeta'''}{(\zeta')^3}-\frac 32 \frac{(\zeta'')^2}{(\zeta')^4}\right) X(\zeta)
+ b^2 \frac{(d-1)(d-3)}{4(2-\tau)^2(\zeta')^2} X(\zeta)\\
&\quad +\ii \frac{b\sigma}{(\zeta')^2} X(\zeta)
- \left(\frac{b}{2-\tau}\right)^{\frac 12 (d-1)(p-1)}(\zeta')^{-\frac {p+3}2} |X(\zeta)|^{p-1} X(\zeta).
\end{align*}
Define $g$, $q$ and $m$ three real-valued functions of $\zeta$ such that for any $\tau$,
\begin{align*}
g(\zeta(\tau))&=\left[\frac{1}2\left(\frac{\zeta'''}{(\zeta')^3}-\frac 32 \frac{(\zeta'')^2}{(\zeta')^4}\right) 
+ \frac{(d-1)(d-3)}{4(2-\tau)^2(\zeta')^2}\right](\tau),\\
q(\zeta(\tau))&= \left[\zeta'(\tau)\right]^{-2},\\
m(\zeta(\tau))&=\left(\frac{1}{2-\tau}\right)^{\frac 12(d-1)(p-1)}\left[\zeta'(\tau)\right]^{-\frac {p+3}2}.
\end{align*}
Then, the equation of $X$ writes, in the $\zeta$ variable,
\begin{equation*}
b^2 X''=\zeta X+ b^2 g(\zeta) X +\ii b\sigma q(\zeta) X - b^{\frac 12(d-1)(p-1)} m(\zeta) |X|^{p-1}X.
\end{equation*}
Moreover, the functions $g$, $q$ and $m$ have the following asymptotics in $\zeta$:
\begin{align}
g(\zeta)& = O\left(|\zeta-\zeta_0|^{-2}\right) \quad \mbox{on $[0,\zeta_0]$},\label{un}\\
g(\zeta)& = O\left(\langle\zeta\rangle^{-2}\right) \quad \mbox{on $(-\infty,0]$},\label{deux}\\
q(\zeta)& = O\left(\langle\zeta\rangle^{-\frac 12}\right) \quad \mbox{on $(-\infty,\zeta_0]$},\label{trois}\\
m(\zeta)& = O\left(|\zeta-\zeta_0|^{-\frac 12(d-1)(p-1)}\right) \quad\mbox{on $[0,\zeta_0]$},\label{quatre}\\
m(\zeta)& = O\left(\langle\zeta\rangle^{-\frac 38(d-1)(p-1)-\frac {p+3}8}\right) \quad \mbox{on $(-\infty,0]$}.\label{cinq}
\end{align}
Last, we set
\begin{align*}
&s=b^{-\frac 23} \zeta,\quad \zeta=b^{\frac 23} s,
\\
&Y(s)=X(\zeta),\quad X''(\zeta)=b^{-\frac 43} Y''(s).
\end{align*}
The equation of $Y$ writes
\begin{equation}\label{eq:Y}
Y''-s Y= G(s,Y)
\end{equation}
where
\begin{equation}\label{def:G}
G(s,Y)= b^{\frac 43} g(b^{\frac 23} s) Y +\ii b^{\frac 13}\sigma q(b^{\frac 23} s) Y - b^{\frac 12(d-1)(p-1)-\frac 23} m(b^{\frac 23} s) |Y|^{p-1}Y.
\end{equation}

Recall that the equation \eqref{eq:U} of $U$ is to be solved on the interval $J=[r_K,r_I]$,
where $r_K=b^{-\frac 12}$ and $r_I=b^{-2}$.
Note first that
\begin{equation*}
r\in J\iff \tau\in [\tau_I,\tau_K]\quad \mbox{where $\tau_I=2-b^{-1}$, $\tau_K=2-b^{\frac 12}$.}
\end{equation*}
Second, 
\begin{equation*}
r\in J\iff \zeta\in [\zeta_I,\zeta_K]
\end{equation*}
where
\begin{align}
\zeta_I&=\zeta(2-b^{-1})= -\frac{3^{\frac 23}}{4} b^{-\frac 43} \left(1+O(b)\right),\label{on:rhoI}\\
\zeta_K&= \zeta(2-b^{\frac 12})=\zeta_0- \zeta_0^{-\frac 12} b^{\frac 12} \left(1+O(b^\frac 12)\right).\label{on:rhoK}
\end{align}
Note also from \eqref{def:zeta}-\eqref{on:rho0} and $\tau_K=2-b^{\frac 12}$ that
\begin{equation}\label{eq:dzd2z}
\zeta'(\tau_K)=\left(\frac{3\pi}4\right)^{-\frac 13}(1+O(b^{\frac 12}))\quad\mbox{and}\quad
\zeta''(\tau_K)=-\frac 12\left(\frac{3\pi}4\right)^{-\frac 43}(1+O(b^{\frac 12})).
\end{equation}
Last, setting
\begin{equation}\label{on:s0}
s_0= b^{-\frac 23} \zeta_0,
\end{equation}
we see that
\begin{equation*}
r\in J\iff s\in M \quad \mbox{where}\quad M=[s_I,s_K],
\end{equation*}
and $s_I$, $s_K$ satisfy
\begin{align}
&s_I= -\frac{3^{\frac 23}}{4} b^{-2} \left(1+O(b)\right),\label{on:sI}\\
&s_K= s_0 - \zeta_0^{-\frac 12} b^{-\frac 16} \left(1+O(b^\frac 12)\right).\label{on:sK}
\end{align}
We also justify for future use that
\begin{equation}\label{pourSk}
\omega(s_K)=\exp\left(\frac 23 s_K^{\frac 32}\right)\leq C \exp\left(\frac {\pi}{2b} -\frac 1{\sqrt{b}} \right)
\leq C \exp\left( -\frac 1{2\sqrt{b}} \right) \Sigma_b^{-1}.
\end{equation}
Indeed, by an elementary Taylor expansion, it holds
\[
\int_0^{2-x}
\sqrt{\tau}\sqrt{4-\tau} d\tau=\pi - 2x+O(x^3),
\]
and thus, using \eqref{on:rhoK}, we obtain the following refinement of the expansion~\eqref{on:sK}
\begin{equation*}
\frac 23 s_K^\frac 32
=\frac 23 b^{-1} \zeta_K^{\frac 32}
=\frac 23 b^{-1} \left[ \zeta(2-b^{\frac 12})\right]^{\frac 32}
=\frac 1{2b} \int_0^{2-b^{\frac 12}} \sqrt{\tau}\sqrt{4-\tau} d\tau
\end{equation*}
which yields
\begin{equation}\label{ee:sKbis}
\frac 23 s_K^\frac 32
=\frac {\pi}{2b}-\frac 1{\sqrt{b}} + O(b^{\frac 12}),
\end{equation}
and proves~\eqref{pourSk}.

Recall from Lemma~\ref{le:Airy} that the functions $\Ai$ and $\BB$ are two independent solutions of the equation $Y''=s Y$ on $\R$ whose
 Wronskian $\Ai\BB'-\Ai'\BB$ was normalized to~$1$.
We use these functions to apply the method of variation of constants to the equation \eqref{eq:Y} of $Y$ on $M$. Let
\begin{equation*}
\left\{\begin{aligned}
Y&=\alpha \Ai + \beta \BB\\
Y'&=\alpha \Ai' + \beta \BB',
\end{aligned}\right.
\end{equation*}
where $\alpha$ and $\beta$ are unknown complex-valued functions.
By standard computation, from \eqref{eq:Y} and \eqref{def:G}, we obtain the system
\begin{equation}\label{syst:ab}\left\{
\begin{aligned}
\alpha' \Ai + \beta' \BB & =0\\
\alpha'\Ai'+\beta' \BB' & = G(s,\alpha\Ai+\beta\BB),
\end{aligned}\right.
\end{equation}
which writes equivalently
\begin{equation}\label{syst:abe}\left\{
\begin{aligned}
\alpha' &= -\BB G(s,\alpha\Ai+\beta\BB)\\
\beta' &= \Ai G(s,\alpha\Ai+\beta\BB).
\end{aligned}\right.
\end{equation}
Let $\alpha_I$, $\beta_I\in \C$ satisfying
\begin{equation}\label{albeI}
|\alpha_I|\leq \Sigma_b,\quad |\beta_I|\leq \Sigma_b, \quad
|\alpha_I|+|\beta_I|\neq 0.
\end{equation}
Set
\begin{align*}
\Gamma_\alpha[\alpha,\beta]&=\alpha_I + \Gamma_{\alpha,1}[\alpha,\beta] + \Gamma_{\alpha,2}[\alpha,\beta] + \Gamma_{\alpha,3}[\alpha,\beta],\\
\Gamma_\beta[\alpha,\beta]&=\beta_I + \Gamma_{\beta,1}[\alpha,\beta] + \Gamma_{\beta,2}[\alpha,\beta] + \Gamma_{\beta,3}[\alpha,\beta],
\end{align*}
where
\begin{align*}
&\Gamma_{\alpha,1}[\alpha,\beta](s)=- \int_{s_I}^s \BB(s')b^{\frac 43} g(b^{\frac 23} s') (\alpha\Ai+\beta\BB) (s') ds',\\
&\Gamma_{\alpha,2}[\alpha,\beta](s)=- \int_{s_I}^s \BB(s')\ii b^{\frac 13}\sigma q(b^{\frac 23} s') (\alpha\Ai+\beta\BB)(s') ds',\\
&\Gamma_{\alpha,3}[\alpha,\beta](s)= \int_{s_I}^s \BB(s')b^{\frac 12(d-1)(p-1)-\frac 23} m(b^{\frac 23} s') |\alpha\Ai+\beta\BB|^{p-1} (\alpha\Ai+\beta\BB)(s') ds',
\end{align*}
and
\begin{align*}
&\Gamma_{\beta,1}[\alpha,\beta](s)=\int_{s_I}^s \Ai(s')b^{\frac 43} g(b^{\frac 23} s') (\alpha\Ai+\beta\BB) (s') ds',\\
&\Gamma_{\beta,2}[\alpha,\beta](s)=\int_{s_I}^s \Ai(s')\ii b^{\frac 13}\sigma q(b^{\frac 23} s') (\alpha\Ai+\beta\BB)(s') ds',\\
&\Gamma_{\beta,3}[\alpha,\beta](s)\\
&\quad = - \int_{s_I}^s \Ai(s')b^{\frac 12(d-1)(p-1)-\frac 23} m(b^{\frac 23} s') |\alpha\Ai+\beta\BB|^{p-1} (\alpha\Ai+\beta\BB)(s') ds'.
\end{align*}
The problem of solving equation \eqref{eq:Y} reduces to find a fixed point $(\alpha,\beta)$ of
$(\Gamma_\alpha,\Gamma_\beta)$.
We work in the following complete metric space
\begin{align*}
E_M=\left\{ (\alpha,\beta):M\to \C^2 \mbox{ is continuous and satisfies $\|(\alpha,\beta)\|_M\leq 1$} \right\}
\end{align*}
equipped with the distance associated to the norm
\[
\|(\alpha,\beta)\|_M=
\sup_{s\in M} \max \left\{ (2|\alpha_I|+b^{\frac 14} |\beta_I|)^{-1} [\omega(s)]^{-2} |\alpha(s)| ;
 (b^{\frac 14}|\alpha_I|+ 2|\beta_I|)^{-1} |\beta(s)|\right\}.
\]
We estimate $\Gamma_{\alpha}$ and $\Gamma_{\beta}$ on the interval $M$ for 
$(\alpha,\beta)\in E_M$, using the asymptotics of Lemma~\ref{le:Airy}.

\emph{Estimate for $\Gamma_{\alpha,1}$.}
First, for $s\in M$, $s\leq 0$, using $\omega(s')=1$ for any $s'\in (-\infty,0]$, the bound $|\Ai \BB|\leq C$ from \eqref{airy3}, \eqref{deux}
and the change of variable $\zeta=b^{\frac 23} s$, it holds
\begin{equation*}
|\Gamma_{\alpha,1}[\alpha,0](s)|\leq 
C b^{\frac 23} (2|\alpha_I|+b^{\frac 14} |\beta_I|) \int_{-\infty}^0 |g(\zeta)| d\zeta
\leq Cb^{\frac 23} (2|\alpha_I|+b^{\frac 14} |\beta_I|) .
\end{equation*}
Similarly, for $s\in M$, $s\leq 0$,
\begin{equation*}
|\Gamma_{\alpha,1}[0,\beta](s)|\leq 
C b^{\frac 23} (2|\beta_I|+b^{\frac 14} |\alpha_I|) .
\end{equation*}
Second, for $s\in M$, $s\geq 0$, using \eqref{airy3} and \eqref{un},
\begin{align*}
|\Gamma_{\alpha,1}[\alpha,0](s)|
&\leq |\Gamma_{\alpha,1}[\alpha,0](0)|+ (2|\alpha_I|+b^{\frac 14} |\beta_I|) \int_0^s \langle s'\rangle^{-\frac 12} |s'-s_0|^{-2} [\omega(s')]^2 ds'\\
&\leq C (2|\alpha_I|+b^{\frac 14} |\beta_I|)\left( b^{\frac 23}+ b^\frac12 [\omega(s)]^2\right)\\
&\leq C b^{\frac 12}(2|\alpha_I|+b^{\frac 14} |\beta_I|)[\omega(s)]^2.
\end{align*}
Indeed, we have used that for $0\leq s'\leq s$, $|s'|\leq \langle s'\rangle$, $\omega(s')\leq \omega(s)$ so that 
\[
\int_0^s \langle s'\rangle^{-\frac 12} |s'-s_0|^{-2} [\omega(s')]^2 ds'
\leq [\omega(s)]^2\int_0^{s_K} |s'|^{-\frac 12} |s'-s_0|^{-2}ds'
\]
and next, by the change of variable $\zeta=b^{\frac 23} s'$,
\begin{equation*}
\int_0^{s_K} | s'|^{-\frac 12} |s'-s_0|^{-2} ds'
=b\int_0^{\zeta_K} |\zeta|^{-\frac 12}|\zeta-\zeta_0|^{-2} d\zeta
\end{equation*}
which yields using \eqref{on:rhoK}
\begin{equation}\label{pour}
\int_0^{s_K} | s'|^{-\frac 12} |s'-s_0|^{-2} ds'
= b \int_0^{\zeta_K/2}+ b\int_{\zeta_K/2}^{\zeta_K}
\leq Cb (1+(\zeta_0-\zeta_K)^{-1})\leq b^{\frac 12}.
\end{equation}
Similarly, for $s\in M$, $s\geq 0$, using \eqref{airy2} and \eqref{pour},
\begin{align*}
|\Gamma_{\alpha,1}[0,\beta](s)|&\leq |\Gamma_{\alpha,1}[0,\beta](0)|
+ (2|\beta_I|+b^{\frac 14} |\alpha_I|) \int_0^s \langle s'\rangle^{-\frac 12} |s'-s_0|^{-2}[\omega(s')]^2 ds'\\
&\leq C b^{\frac 12}(2|\beta_I|+b^{\frac 14} |\alpha_I|)[\omega(s)]^2.
\end{align*}
It follows that for all $s\in M$,
\begin{equation}\label{on:Galpha1}
|\Gamma_{\alpha,1}[\alpha,\beta](s)|
\leq C b^{\frac 14} (2|\alpha_I|+b^{\frac 14} |\beta_I|)[\omega(s)]^2.
\end{equation}

\emph{Estimate for $\Gamma_{\beta,1}$.}
First, arguing as for $\Gamma_{\alpha,1}$ we have, for $s\in M$, $s\leq 0$,
\begin{equation*}
|\Gamma_{\beta,1}[\alpha,0](s)|
\leq C b^{\frac 23} (2|\alpha_I|+b^{\frac 14} |\beta_I|),
\end{equation*}
and
\begin{equation*}
|\Gamma_{\beta,1}[0,\beta](s)|\leq 
C b^{\frac 23} (2|\beta_I|+b^{\frac 14} |\alpha_I|) .
\end{equation*}
Second, for $s\in M$, $s\geq 0$, using \eqref{airy1} and \eqref{pour}
\begin{align*}
|\Gamma_{\beta,1}[\alpha,0](s)|
&\leq |\Gamma_{\beta,1}[\alpha,0](0)|+ (2|\alpha_I|+b^{\frac 14} |\beta_I|) \int_0^{s_K} \langle s'\rangle^{-\frac 12} |s'-s_0|^{-2} ds'\\
&\leq C b^{\frac 12}(2|\alpha_I|+b^{\frac 14} |\beta_I|).
\end{align*}
and using \eqref{airy3} and \eqref{pour}
\begin{align*}
|\Gamma_{\beta,1}[0,\beta](s)|
&\leq |\Gamma_{\beta,1}[0,\beta](0)|+ (2|\beta_I|+b^{\frac 14} |\alpha_I|) \int_0^{s_K} \langle s'\rangle^{-\frac 12} |s'-s_0|^{-2}ds'\\
&\leq C b^{\frac 12}(2|\beta_I|+b^{\frac 14} |\alpha_I|).
\end{align*}
It follows that for all $s\in M$,
\begin{equation}\label{on:Gbeta1}
|\Gamma_{\beta,1}[\alpha,\beta](s)|
\leq C b^{\frac 14}(2|\beta_I|+b^{\frac 14} |\alpha_I|).
\end{equation}

We deduce from \eqref{on:Galpha1}-\eqref{on:Gbeta1} and the definition of $\|\cdot\|_M$ that, for any $(\alpha,\beta),$ $(\tilde \alpha,\tilde \beta)\in E_M$,
\begin{equation*}
\|(\Gamma_{\alpha,1},\Gamma_{\beta,1})[\alpha,\beta]\|_M\leq C b^{\frac 14}.
\end{equation*}
Therefore, by linearity of $(\alpha,\beta)\mapsto (\Gamma_{\alpha,1},\Gamma_{\beta,1})[\alpha,\beta]$,
it hold
\begin{equation}\label{on:G1}
\|(\Gamma_{\alpha,1},\Gamma_{\beta,1})[\alpha,\beta]\|_M\leq C b^{\frac 14}\|(\alpha,\beta)\|_M,
\end{equation}
and
\begin{equation}\label{on:G1c}
\|(\Gamma_{\alpha,1},\Gamma_{\beta,1})[\alpha,\beta] - (\Gamma_{\alpha,1},\Gamma_{\beta,1})[\tilde \alpha,\tilde\beta]\|_M
\leq C b^{\frac 14}\|(\alpha,\beta)-(\tilde\alpha,\tilde\beta)\|_M.
\end{equation}

\emph{Estimate for $\Gamma_{\alpha,2}$.}
First, for $s\in M$ $s\leq 0$, using \eqref{airy2}, \eqref{airy3} and \eqref{trois}, one has
\begin{align*}
|\Gamma_{\alpha,2}[\alpha,0](s)|
&\leq C \sigma (2|\alpha_I|+b^{\frac 14} |\beta_I|) b^{\frac 13} \int_{s_I}^0 \langle s\rangle^{-\frac 12} \langle b^{\frac 23}s\rangle^{-\frac 12} ds\\
&\leq C \sigma (2|\alpha_I|+b^{\frac 14} |\beta_I|) \int_{s_I}^0 \langle s\rangle^{-1} ds
 \leq C \sigma |\ln b|(2|\alpha_I|+b^{\frac 14} |\beta_I|)
\end{align*}
(we have used $\langle b^{\frac 23} s\rangle \geq b^{\frac 23}\langle s\rangle$).
Similarly,
\begin{equation*}
|\Gamma_{\alpha,2}[0,\beta](s)|
\leq C \sigma |\ln b|(2|\beta_I|+b^{\frac 14} |\alpha_I|).
\end{equation*}
Second, for $s\in M$, $s\geq 0$, using \eqref{airy3},
\begin{align*}
|\Gamma_{\alpha,2}[\alpha,0](s)|
&\leq |\Gamma_{\alpha,2}[\alpha,0](0)|
+ (2|\alpha_I|+b^{\frac 14} |\beta_I|) \sigma b^{\frac 13} \int_0^s \langle s'\rangle^{-\frac 12} [\omega(s')]^2 ds'\\
&\leq C (2|\alpha_I|+b^{\frac 14} |\beta_I|)\left(\sigma |\ln b|+ \sigma [\omega(s)]^2\right)\\
&\leq C \sigma |\ln b| (2|\alpha_I|+b^{\frac 14} |\beta_I|)[\omega(s)]^2,
\end{align*}
where we have used, for $0\leq s\leq s_K\leq C b^{-\frac 23}$,
\begin{equation}\label{pour2}
\int_0^s \langle s'\rangle^{-\frac 12} [\omega(s')]^2 ds'
 \leq [\omega(s)]^2 \int_0^s \langle s'\rangle^{-\frac 12} ds'
\leq C \langle s\rangle^{\frac 12}[\omega(s)]^2\leq C b^{-\frac 13} [\omega(s)]^2 .
\end{equation}
Similarly, for $s\in M$, $s\geq 0$, using \eqref{airy3} and \eqref{pour2},
\begin{align*}
|\Gamma_{\alpha,2}[0,\beta](s)|
&\leq |\Gamma_{\alpha,2}[0,\beta](0)|
+ (2|\beta_I|+b^{\frac 14} |\alpha_I|) \sigma b^{\frac 13} \int_0^s \langle s'\rangle^{-\frac 12} [\omega(s')]^2 ds'\\
&\leq C \sigma |\ln b| (2|\beta_I|+b^{\frac 14} |\alpha_I|)[\omega(s)]^2.
\end{align*}
It follows that for all $s\in M$,
\begin{equation}\label{on:Galpha2}
|\Gamma_{\alpha,2}[\alpha,\beta](s)|
\leq C \sigma b^{-\frac 12} (2|\alpha_I|+b^{\frac 14} |\beta_I|)[\omega(s)]^2.
\end{equation}

\emph{Estimate for $\Gamma_{\beta,2}$.} First, arguing as for $\Gamma_{\alpha,2}$, for $s\in M$ $s\leq 0$, one has using 
\eqref{airy1}, \eqref{airy3} and \eqref{trois},
\begin{equation*}
|\Gamma_{\beta,2}[\alpha,0](s)|
\leq C \sigma |\ln b|(2|\alpha_I|+b^{\frac 14} |\beta_I|)
\end{equation*}
and
\begin{equation*}
|\Gamma_{\beta,2}[0,\beta](s)|\leq C \sigma |\ln b|(2|\beta_I|+b^{\frac 14} |\alpha_I|).
\end{equation*}
Second, for $s\in M$ $s\geq 0$, using \eqref{airy1} and \eqref{airy3},
\begin{align*}
|\Gamma_{\beta,2}[\alpha,0](s)|
&\leq |\Gamma_{\beta,2}[\alpha,0](0)|
+ (2|\alpha_I|+b^{\frac 14} |\beta_I|) \sigma b^{\frac 13} \int_0^s \langle s'\rangle^{-\frac 12} ds'\\
&\leq C (2|\alpha_I|+b^{\frac 14} |\beta_I|)\left(\sigma |\ln b|+ \sigma b^\frac13 s_K^{\frac 12} \right)\\
&\leq C \sigma |\ln b| (2|\alpha_I|+b^{\frac 14} |\beta_I|),
\end{align*}
and similarly,
\begin{align*}
|\Gamma_{\beta,2}[0,\beta](s)|
&\leq |\Gamma_{\beta,2}[0,\beta](0)|
+ (2|\beta_I|+b^{\frac 14} |\alpha_I|) \sigma b^{\frac 13} \int_0^s \langle s'\rangle^{-\frac 12} ds'\\
&\leq C \sigma |\ln b| (2|\beta_I|+b^{\frac 14} |\alpha_I|).
\end{align*}
It follows that for all $s\in M$,
\begin{equation}\label{on:Gbeta2}
|\Gamma_{\beta,2}[\alpha,\beta](s)|
\leq C \sigma b^{-\frac 14}|\ln b| (2|\beta_I|+b^{\frac 14} |\alpha_I|)
\leq C \sigma b^{-\frac 12} (2|\beta_I|+b^{\frac 14} |\alpha_I|).
\end{equation}
By linearity of $(\alpha,\beta)\mapsto (\Gamma_{\alpha,2},\Gamma_{\beta,2})[\alpha,\beta]$, we deduce from \eqref{on:Galpha2}-\eqref{on:Gbeta2}, that, for any $(\alpha,\beta), (\tilde \alpha,\tilde \beta)\in E_M$,
\begin{equation}\label{on:G2}
\|(\Gamma_{\alpha,2},\Gamma_{\beta,2})[\alpha,\beta]\|_M\leq C \sigma b^{-\frac 12} \|(\alpha,\beta)\|_M,
\end{equation}
and\begin{equation}\label{on:G2c}
\|(\Gamma_{\alpha,2},\Gamma_{\beta,2})[\alpha,\beta] - (\Gamma_{\alpha,2},\Gamma_{\beta,2})[\tilde \alpha,\tilde\beta]\|_M
\leq C \sigma b^{-\frac 12} \|(\alpha,\beta)-(\tilde\alpha,\tilde\beta)\|_M.
\end{equation}

\emph{Estimate for $\Gamma_{\alpha,3}$.}
First, we note that for any $(\alpha,\beta)\in E_M$, by \eqref{airy1} and then \eqref{albeI}, \eqref{pourSk},
it holds on $M$,
\begin{equation*}
|\alpha \Ai|\leq C (2|\alpha_I|+b^\frac 14 |\beta_I|) \langle s \rangle^{-\frac 14}\omega(s)\\
\leq C \Sigma_b \omega(s_K) \leq C\exp\left( -\frac 1{2\sqrt{b}}\right).
\end{equation*}
Similarly, by \eqref{airy2}, and then \eqref{albeI}, \eqref{pourSk},
\begin{equation*}
|\beta \BB|\leq C (2|\beta_I|+b^\frac 14 |\alpha_I|) \langle s \rangle^{-\frac 14}\omega(s)\\
\leq C \Sigma_b \omega(s_K) \leq C\exp\left( -\frac 1{2\sqrt{b}}\right).
\end{equation*}
These estimates imply the following uniform estimate on $M$,
\begin{equation}\label{sur:Yunif}
\left(|\alpha \Ai |+|\beta \BB|\right)^{p-1}\leq C\exp\left( -\frac {p-1}{2\sqrt{b}}\right).
\end{equation}

From \eqref{quatre}-\eqref{cinq} and \eqref{on:sI}-\eqref{on:sK}, for all $s\in M$, it holds
\begin{equation}\label{pour3}
b^{\frac 12(d-1)(p-1)-\frac 23} |m(b^{\frac 23} s)|\leq C b^{\frac 14 (d-1)(p-1)-\frac 23}\leq C b^{-\frac 23}.
\end{equation}
Thus, using also \eqref{airy2}-\eqref{airy3}, again \eqref{on:sI}-\eqref{on:sK}, and \eqref{sur:Yunif},
\begin{align*}
|\Gamma_{\alpha,3}[\alpha,0|(s)|
&\leq C(2|\alpha_I|+b^{\frac 14} |\beta_I|) \exp\left( -\frac {p-1}{2\sqrt{b}}\right) b^{-\frac 23} \int_{s_I}^s [\omega(s')|^2 ds'\\
&\leq C(2|\alpha_I|+b^{\frac 14} |\beta_I|) \exp\left( -\frac {p-1}{2\sqrt{b}}\right) b^{-\frac 83} [\omega(s)]^2,
\end{align*}
and
\begin{align*}
|\Gamma_{\alpha,3}[0,\beta|(s)|
&\leq C (2|\beta_I|+b^{\frac 14} |\alpha_I|) \exp\left( -\frac {p-1}{2\sqrt{b}}\right) b^{-\frac 23} \int_{s_I}^s [\omega(s')|^2 ds'\\
&\leq C (2|\beta_I|+b^{\frac 14} |\alpha_I|) \exp\left( -\frac {p-1}{2\sqrt{b}}\right) b^{-\frac 83} [\omega(s)]^2.
\end{align*}
It follows that for any $s\in M$,
\begin{equation}\label{on:Galpha3}
|\Gamma_{\alpha,3}[\alpha,\beta](s)|
\leq C\exp\left( -\frac {p-1}{4\sqrt{b}}\right) (2|\alpha_I|+b^{\frac 14} |\beta_I|)[\omega(s)]^2.
\end{equation}

\emph{Estimate for $\Gamma_{\beta,3}$.}
Similarly, using \eqref{airy1}, \eqref{airy3}, \eqref{pour3} and \eqref{sur:Yunif}, one has,
\begin{align*}
|\Gamma_{\beta,3}[\alpha,0|(s)|
&\leq C(2|\alpha_I|+b^{\frac 14} |\beta_I|) \exp\left( -\frac {p-1}{2\sqrt{b}}\right) b^{-\frac 23} \int_{s_I}^{s_K} ds'\\
&\leq C(2|\alpha_I|+b^{\frac 14} |\beta_I|) \exp\left( -\frac {p-1}{2\sqrt{b}}\right) b^{-\frac 83} ,
\end{align*}
and 
\begin{align*}
|\Gamma_{\beta,3}[0,\beta|(s)|
&\leq C (2|\beta_I|+b^{\frac 14} |\alpha_I|) \exp\left( -\frac {p-1}{2\sqrt{b}}\right) b^{-\frac 23} \int_{s_I}^{s_K} ds'\\
&\leq C (2|\beta_I|+b^{\frac 14} |\alpha_I|) \exp\left( -\frac {p-1}{2\sqrt{b}}\right) b^{-\frac 83}.
\end{align*}
It follows that for any $s\in M$,
\begin{equation}\label{on:Gbeta3}
|\Gamma_{\beta,3}[\alpha,\beta](s)|
\leq C \exp\left( -\frac {p-1}{4\sqrt{b}}\right) (2|\beta_I|+b^{\frac 14} |\alpha_I|).
\end{equation}

We deduce from \eqref{on:Galpha3}-\eqref{on:Gbeta3}, for any $(\alpha,\beta)\in E_M$,
by homogeneity,
\begin{equation}\label{on:G3}
\|(\Gamma_{\alpha,3},\Gamma_{\beta,3})[\alpha,\beta]\|_M\leq C \exp\left( -\frac {p-1}{4\sqrt{b}}\right) \|(\alpha,\beta)\|_M^{p}.
\end{equation}
For any $(\alpha,\beta), (\tilde \alpha,\tilde \beta)\in E_M$,
using the inequality 
\[||Y|^{p-1}Y-|\tilde Y|^{p-1}\tilde Y|\leq C (|Y|^{p-1}+|\tilde Y|^{p-1}) |Y-\tilde Y|,\]
and the above estimates (in particular, \eqref{sur:Yunif} for $(\alpha,\beta)$ and $(\tilde \alpha,\tilde \beta)$ in $E$),
we obtain
\begin{multline}\label{on:G3c}
 \|(\Gamma_{\alpha,3},\Gamma_{\beta,3})[\alpha,\beta] - (\Gamma_{\alpha,3},\Gamma_{\beta,3})[\tilde \alpha,\tilde\beta]\|_M
\\ 
\leq C \exp\left( -\frac {p-1}{4\sqrt{b}}\right)\left(\|(\alpha,\beta)\|_M^{p-1}+\|(\tilde\alpha,\tilde\beta)\|_M^{p-1}\right)
 \|(\alpha,\beta)-(\tilde\alpha,\tilde\beta)\|_M.
\end{multline}

Using the definitions of $\omega(s)$ and of the norm $\|\cdot\|_M$, one has $\|(\alpha_I,\beta_I)\|_M\leq \frac 12$.
Combining \eqref{on:G1}, \eqref{on:G2} and \eqref{on:G3}, for any $(\alpha,\beta)\in E_M$, we have
\[\|(\Gamma_{\alpha},\Gamma_{\beta})[\alpha,\beta]\|_M\leq \frac 12 + C (b^{\frac 14}+\sigma b^{-\frac 12}).\]
Thus, using \eqref{SM}, for $\sigma$ small enough, $(\Gamma_\alpha,\Gamma_\beta)$ maps $E_M$ into $E_M$.

Similarly, using \eqref{on:G1c}, \eqref{on:G2c} and \eqref{on:G3c}, 
for any $(\alpha,\beta), (\tilde \alpha,\tilde \beta)\in E_M$, we have
\begin{equation*}
\|(\Gamma_{\alpha},\Gamma_{\beta})[\alpha,\beta]-(\Gamma_{\alpha},\Gamma_{\beta})[\tilde\alpha,\tilde\beta]\|_M
\leq C (b^{\frac 14}+\sigma b^{-\frac 12})
\|(\alpha,\beta)-(\tilde\alpha,\tilde\beta)\|_M,
\end{equation*}
and thus, for $b$ small enough, $(\Gamma_\alpha,\Gamma_\beta)$ is a contraction on $E_M$.

The application $(\Gamma_\alpha,\Gamma_\beta):E_M\to E_M$ thus admits a unique fixed point $(\alpha,\beta)$ on $E_M$.
The corresponding continuous functions $\alpha$ and $\beta$ satisfy, for all $s\in M$,
\begin{equation*}\left\{\begin{aligned}
\alpha(s)&=\alpha_I - \int_{s_I}^s \BB(s') G\left(s',\alpha(s')\Ai(s')+\beta(s')\BB(s')\right)ds'\\
\beta(s)&=\beta_I + \int_{s_I}^s \Ai(s') G\left(s',\alpha(s')\Ai(s')+\beta(s')\BB(s')\right)ds'.
\end{aligned}\right.
\end{equation*}
In particular, the functions $\alpha$ and $\beta$ are of class $\mathcal C^1$ on $M$, and satisfy~\eqref{syst:abe} and equivalently~\eqref{syst:ab}.

Therefore, for any $\alpha_I$, $\beta_I\in \C$ satisfying \eqref{albeI},
the function $Y=\alpha \Ai+\beta\BB$ satisfies \eqref{eq:Y}, with 
\begin{align*}
Y(s_I)&=\alpha_I \Ai(s_I)+\beta_I\BB(s_I),\\
Y'(s_I)&=\alpha_I \Ai'(s_I)+\beta_I\BB'(s_I).
\end{align*}

Consider the solution $Y$ constructed above for given $\alpha_I$ and $\beta_I$ satisfying \eqref{albeI}, and the corresponding function $U$ 
obtained from $Y$ by the above changes of variables. The following relation 
between $Y(s_I)$ and $U(r_I)$ holds
\begin{equation}\label{yI}
Y(s_I)= [\zeta'(\tau_I)]^{\frac12}{U(r_I)}.
\end{equation}
To relate $Y'(s_I)$ to $U(r_I)$ and $U'(r_I)$, we first observe from $X(\zeta(\tau))=\sqrt{\zeta'(\tau)}W(\tau)$
that
\[
\zeta'(\tau)X'(\zeta(\tau))= \sqrt{\zeta'(\tau)}W'(\tau)+\frac 12 \frac{\zeta''(\tau)}{\sqrt{\zeta'(\tau)}} W(\tau).
\]
Using also $Y'(s_I)=b^{\frac 23} X'(\zeta_I)$, $W(\tau_I)=U(r_I)$ and $-b W'(\tau_I)=U'(r_I)$, we obtain
\begin{align}
Y'(s_I)&=b^{\frac 23}[\zeta'(\tau_I )]^{-\frac 12} W'(\tau_I) +\frac 12 b^{\frac 23}\zeta''(\tau_I)[\zeta'(\tau_I)]^{-\frac 32} W(\tau_I)\nonumber\\
&=-b^{-\frac 13}[\zeta'(\tau_I )]^{-\frac 12} U'(r_I) +\frac 12 b^{\frac 23}\zeta''(\tau_I)[\zeta'(\tau_I)]^{-\frac 32} U(r_I).
\label{dyI}
\end{align}

Let now $u_I$, $\tilde u_I\in \C$ satisfy \eqref{on:uIduI}. From \eqref{yI} and \eqref{dyI}, it is natural to set
\begin{equation}\label{yuI}\left\{\begin{aligned}
y_I&=[\zeta'(\tau_I)]^{\frac12} u_I\\
\tilde y_I&=-b^{-\frac 13}[\zeta'(\tau_I )]^{-\frac 12} \tilde u_I +\frac 12 b^{\frac 23}\zeta''(\tau_I)[\zeta'(\tau_I)]^{-\frac 32} u_I, 
\end{aligned}\right.\end{equation}
so that
\[
(U(r_I),U'(r_I))=(u_I,\tilde u_I)\iff
(Y(s_I),Y'(s_I))=(y_I,\tilde y_I)
.
\]
Next, we have from the system
\begin{align*}
Y(s_I)&=\alpha_I \Ai(s_I) + \beta_I \BB(s_I)\\
Y'(s_I)&=\alpha_I \Ai'(s_I) + \beta_I \BB'(s_I),
\end{align*}
and \eqref{wrAiB}, 
the relation
\begin{equation}\label{abyI}
(Y(s_I),Y'(s_I))=(y_I,\tilde y_I)\iff
\left\{\begin{aligned}
\alpha_I&= y_I \BB'(s_I)-\tilde y_I \BB(s_I)\\
\beta_I&=-y_I \Ai'(s_I)+\tilde y_I \Ai(s_I).
\end{aligned}\right.\end{equation}
Using the following estimates (consequences of $\tau_I=2-b^{-1}$, \eqref{on:rho4}-\eqref{on:rho5} and Lemma~\ref{le:Airy})
\begin{align*}
&|\Ai(s_I)|+|\BB(s_I)|\leq C b^{\frac 12},\quad 
|\Ai'(s_I)|+|\BB'(s_I)|\leq C b^{-\frac 12},\\
&|\zeta'(\tau_I)|^{-\frac 12}\leq C b^{\frac 16},\quad
|\zeta''(\tau_I)|\leq C b^{\frac 23},
\end{align*}
we have 
\begin{equation*}
|y_I|\leq C b^{-\frac 16} |u_I|,\quad
|\tilde y_I|\leq C b^{-\frac 16} \left(|\tilde u_I|+b |u_I|\right),
\end{equation*}
and thus
\begin{align*}
|\alpha_I|& \leq |y_I| |\BB'(s_I)|+|\tilde y_I \BB(s_I)|
\leq C (b^{-\frac 12} |y_I|+b^{\frac 12} |\tilde y_I|)
\leq C b^{-\frac 23} \left(|u_I|+b|\tilde u_I|\right)
\\
|\beta_I|& \leq |y_I| |\Ai'(s_I)|+|\tilde y_I \Ai(s_I)|
\leq C b^{-\frac 23} \left(|u_I|+b|\tilde u_I|\right).
\end{align*}
Using also \eqref{on:uIduI}, we obtain for $b$ small enough,
\begin{equation*}
|\alpha_I|+|\beta_I| \leq C b^{\frac 13} \Sigma_b<\Sigma_b,
\end{equation*}
which means that \eqref{albeI} is satisfied.
Therefore, we have proved the existence of a solution~$U_J$ 
satisfying equation \eqref{eq:U} on $J$ and the conditions \eqref{ini:U_J} at $r=r_I$.

The continuity of $U_J$ in $(\sigma,b,u_I,\tilde u_I,r)$ follows from standard arguments using the
Fixed-Point Theorem with parameter.

\medskip

\textbf{Step 2.} 
Here, we construct admissible solutions of \eqref{eq:U} on $J\cup I=[r_K,\infty)$, and describe their behavior at the matching point $r_K=b^{-\frac 12}$.
Let $U$ be a solution of \eqref{eq:U} constructed in Proposition~\ref{pr:UonI} on $I$.
In order to use the construction of Step 1 with $u_I=U(r_I)$ and $\tilde u_I=U'(r_I)$, it suffices to check that
\eqref{on:uIduI} holds, \emph{i.e.}
\begin{equation*}
|U(r_I)|\leq b \Sigma_b,\quad |U'(r_I)|\leq \Sigma_b.
\end{equation*}
But this is a direct consequence of \eqref{on:U}-\eqref{on:dU} and \eqref{on:b}-\eqref{on:rho}, for $\sigma$ small enough. Now, we will prove the existence of a real $\theta_\ext\in [0,2\pi)$ 
such that the function $U_\ext$ defined by
\[
U_\ext(r)=\ee^{\ii \theta_\ext} U(r) \quad \mbox{for any $r\in J\cup I$,}
\]
satisfies
\begin{align}
 \Re(U_\ext(r_K))&=\frac{\rho \sqrt{b}}{\sqrt{2}}
\exp\left(\frac{\pi}{2b}\right) \exp\left(-\frac1{\sqrt{b}}\right) (1+O(b^{\frac 14})),\label{Uext1}\\
 \Re(U_\ext'(r_K))&= -\frac{\rho \sqrt{b}}{\sqrt{2}}\exp\left(\frac{\pi}{2b}\right) \exp\left(-\frac1{\sqrt{b}}\right) (1+O(b^{\frac 14})),\label{Uext2}\\
 \Im(U_\ext(r_K))&=\frac{\rho \sqrt{b}}{2\sqrt{2}} 
\exp\left(-\frac{\pi}{2b}\right) \exp\left(\frac1{\sqrt{b}}\right) (1+O(b^{\frac 14}))\label{Uext3} ,\\
 \Im(U_\ext'(r_K))&= \frac{\rho \sqrt{b}}{2\sqrt{2}}\exp\left(-\frac{\pi}{2b}\right) \exp\left(\frac1{\sqrt{b}}\right) (1+O(b^{\frac 14})).\label{Uext4}
\end{align}
It is clear by phase invariance that $U_\ext$ is solution of \eqref{eq:U} on $J\cup I$.
 
From the specific behavior \eqref{on:U}-\eqref{on:dU} of the solution $U$ at $r_I$ and the relations 
\eqref{yuI} and \eqref{abyI}, we give refined information on $\alpha_I$ and $\beta_I$
and on the solution $Y$ on $[s_I,s_K]$ (we follow the notation introduced in Step~1).
\begin{lemma} It holds
\begin{equation}\label{eq:aIbI}
|\beta_I|=\frac 1{\sqrt{2\pi}} \rho b^{\frac 13} \left(1+O(b|\ln b|)\right),
\quad 
|\alpha_I|\leq C \rho b^{\frac 43}.
\end{equation}
Moreover, for all $s\in [s_I,s_K]$,
\begin{equation}\label{Yens}
Y(s)=\beta_I \BB(s) \left(1+O(b^{\frac 14})\right),\quad
Y'(s)=\beta_I \BB'(s)\left(1+O(b^{\frac 14})\right),
\end{equation}
where $s_I$, $s_K$ are defined in \eqref{on:sI}, \eqref{on:sK}.
\end{lemma}
\begin{proof}
First, note that from \eqref{on:U}-\eqref{on:UdU}, $r_I=b^{-2}$ and \eqref{on:b}, we have 
\begin{align}
u_I=U(r_I) & 
=\rho b \exp\left( \frac{\ii}{4b^3}\right) \exp\left(\frac{2\ii\ln b}{b}\right)\left(1+O(b|\ln b|)\right),\label{on:UI}\\
\tilde u_I=U'(r_I) & = \frac \ii 2 \rho \exp\left( \frac{\ii}{4b^3}\right) \exp\left(\frac{2\ii\ln b}{b}\right)\left(1+O(b|\ln b|)\right),\label{on:dUI}\\
\tilde u_I-\frac\ii2 b^{-1} u_I& = O(\rho b^2).\label{on:UIdUI}
\end{align}
Second, from \eqref{on:rho4}-\eqref{on:rho5} and $\tau_I=2-b^{-1}$, one has
\[
[\zeta'(\tau_I)|^{\frac 12}=3^{-\frac 16} b^{-\frac 16} (1+O(b)),\quad
\zeta''(\tau_I)=-3^{-\frac 43} b^{\frac 23}(1+O(b)).
\]
Inserting in \eqref{yuI}, this gives
\begin{equation}\label{newyuI}\left\{\begin{aligned}
y_I&=3^{-\frac 16} b^{-\frac 16} u_I (1+O(b)) \\
\tilde y_I&= -3^{\frac 16}b^{-\frac 16} \tilde u_I(1+O(b))-\frac 12 3^{-\frac 56}b^{\frac{11}6}u_I (1+O(b)).
\end{aligned}\right.\end{equation}
For $\alpha_I$, using \eqref{abyI}, Lemma~\ref{le:Airy} and next $|s_I|^{\frac 12}=\frac 12 3^{\frac 13} b^{-1}(1+O(b))$,
this yields
\begin{align*}
\alpha_I&= y_I \BB'(s_I)-\tilde y_I \BB(s_I)
=\left(-\ii |s_I|^{\frac 12} y_I (1+O(b^{3})) -\tilde y_I \right)\BB(s_I)\\
&=3^{\frac 16}b^{-\frac 16}\left(-\frac \ii2 b^{-1} u_I(1+O(b))
+ \tilde u_I(1+O(b))\right)\BB(s_I).
\end{align*}
Using now \eqref{on:UI}-\eqref{on:UIdUI} and $\BB(s_I)=O(b^{\frac 12})$ (from Lemma~\ref{le:Airy}) we obtain
\[
\alpha_I=O\left(\rho b^{\frac 43}\right).
\]
For $\beta_I$, using \eqref{abyI}, Lemma~\ref{le:Airy} and next $|s_I|^{\frac 12}=\frac 12 3^{\frac 13} b^{-1}(1+O(b))$,
we have
\begin{align*}
\beta_I&=-y_I \Ai'(s_I)+\tilde y_I \Ai(s_I)
= \frac 12 3^{\frac 13}b^{-1}y_I\Im(\xi_I)(1+O(b))+\tilde y_I\Re(\xi_I),
\end{align*}
where we have set
\begin{align*}
\xi_I&=\frac{1}{\sqrt{\pi}} |s_I|^{-\frac 14} \ee^{\ii \frac{\pi}4} \exp\left(-\frac 23\ii|s_I|^{\frac 32}\right),\\
|\xi_I|&=\frac{1}{\sqrt{\pi}} |s_I|^{-\frac 14}
=\frac 1{\sqrt{\pi}} 2^{\frac 12} 3^{-\frac 16} b^{\frac 12}(1+O(b)).
\end{align*}
Using \eqref{newyuI}, we obtain
\begin{equation*}
\beta_I=
\frac 12 3^{\frac 16}b^{-\frac 76} \left( \Im(\xi_I) (1+O(b))+ \Re(\xi_I) O(b^3)\right) u_I
-3^{\frac 16}b^{-\frac 16} \Re(\xi_I) (1+O(b)) \tilde u_I .
\end{equation*}
We continue using \eqref{on:UI}-\eqref{on:dUI},
\begin{align*}
\beta_I
&=\frac 12 3^{\frac 16}\rho b^{-\frac 16} \exp\left( \frac{\ii}{4b^3}\right) \exp\left(\frac{2\ii\ln b}{b}\right) \\&\quad 
\times \left[ \Im(\xi_I) \left(1+O(b|\ln b|)\right) - \ii \Re(\xi_I) \left(1+O(b|\ln b|)\right) \right]\\
&=-\frac 12 3^{\frac 16}\rho b^{-\frac 16} \exp\left( \frac{\ii}{4b^3}\right) \exp\left(\frac{2\ii\ln b}{b}\right) \ii \xi_I\left(1+O(b|\ln b|)\right).
\end{align*}
In particular,
\begin{equation*}
|\beta_I|=
\frac 12 3^{\frac 16}\rho b^{-\frac 16} |\xi_I|\left(1+O(b|\ln b|)\right)
=\frac 1{\sqrt{2\pi}} \rho b^{\frac 13} \left(1+O(b|\ln b|)\right),
\end{equation*}
which completes the proof of \eqref{eq:aIbI}. Observe that by \eqref{on:b} and \eqref{on:rho},
$|\beta_I|\neq 0$.

Proof of \eqref{Yens}.
It follows from estimates~\eqref{on:Galpha1}, \eqref{on:Galpha2}, \eqref{on:Galpha3} 
and~\eqref{on:Gbeta1}, \eqref{on:Gbeta2}, \eqref{on:Gbeta3}) that, for all $s\in [s_I,s_K]$,
\begin{align*}
\alpha(s)&=\alpha_I+O\left((|\alpha_I|+b^{\frac 14}|\beta_I|)(b^{\frac 14}+\sigma b^{-\frac 12})[\omega(s)]^2\right),\\
\beta(s)&=\beta_I+O\left((b^{\frac 14}|\alpha_I|+|\beta_I|)(b^{\frac 14}+\sigma b^{-\frac 12})\right).
\end{align*}
Using~\eqref{on:rho}, \eqref{SM} and \eqref{eq:aIbI}, for all $s\in [s_I,s_K]$, we obtain
\begin{equation*}
|\alpha(s)| \leq C \rho b^{\frac 56} [\omega(s)]^2\leq Cb^{\frac 12} |\beta_I| [\omega(s)]^2,\quad
|\beta(s)-\beta_I| \leq C b^{\frac 14}|\beta_I| .
\end{equation*}
Going back to $Y=\alpha \Ai+\beta \BB$ and $Y'=\alpha \Ai'+\beta \BB'$, since by Lemma~\ref{le:Airy},
$\omega^2 |\Ai|\leq C |\BB| $ and $\omega^2 |\Ai'|\leq C|\BB'|$ on $\R$, we obtain~\eqref{Yens}.
\end{proof}

 In view of the definition of $G$ in \eqref{def:G}, we set
\begin{equation*}
k(s) = s + b^{\frac 43} g(b^{\frac 23} s) - b^{\frac 12(d-1)(p-1)-\frac 23} m(b^{\frac 23} s) |Y(s)|^{p-1},
\end{equation*}
and
\begin{equation}\label{def:l}
\ell (s) = b^{\frac 13}\sigma q(b^{\frac 23} s).
\end{equation}
Note that the functions $k$ and $\ell$ are real-valued.
The equation \eqref{eq:Y} of $Y$ writes
\[
Y'' - k Y = \ii \ell Y.
\]
Now, we set (recall that $\beta_I\neq 0$ by \eqref{eq:aIbI})
\begin{equation}\label{def:Z}
Z(s)=\beta_I^{-1} Y(s),
\end{equation}
so that $Z$ is also solution of the linear equation
\begin{equation}\label{eq:Z}
Z''-k Z=\ii \ell Z,
\end{equation}
and satisfies on $[s_I,s_K]$,
\begin{equation}\label{on:Z}
Z(s)=\BB(s) \left(1+O(b^{\frac 14})\right),\quad
Z'(s)= \BB'(s)\left(1+O(b^{\frac 14})\right).
\end{equation}
We also set
\[
R(s)=\Re(Z(s)),
\]
solution of the linear equation
\begin{equation}\label{eq:R}
R''-k R= - \ell \Im(Z),
\end{equation}
and satisfing on $[s_I,s_K]$,
\begin{equation}\label{on:R}
R(s)=\Re(\BB(s)) \left(1+O(b^{\frac 14})\right),\quad
R'(s)=\Re(\BB'(s))\left(1+O(b^{\frac 14})\right).
\end{equation}
Last, we set for $s\in [s_I,s_K]$,
\[
T(s)= - Z(s) \int_s^{s_K} \frac{ds'}{Z^2(s')},\quad
S(s)= \Re (T(s)).
\]
(Recall that from \eqref{on:Z} and Lemma~\ref{le:Airy}, $Z$ does not vanish on $[s_I,s_K]$.)
Since $T$ is solution of $T''-k T=i\ell T$, $S$ is solution of
\begin{equation}\label{eq:S}
S''-k S = -\ell \Im(T).
\end{equation}
Moreover, by the definition of $T$ and $S$, it follows that
\begin{equation*}
\mathcal W(Z,T)=Z T'-Z'T=1.
\end{equation*}

\begin{lemma} 
It holds, on $[s_I,s_K]$,
\begin{equation}\label{eq:WrRS}
\mathcal W(R,S)=RS'-R'S= 1+O(b^{\frac 14}).
\end{equation}
\end{lemma}
\begin{proof}
From $S(s_K)=0$ and $S'(s_K)=\Re(\frac1{Z(s_K)})=\frac{\Re(Z(s_K))}{|Z(s_K)|^2}$, we compute
\[
\mathcal W(R,S)(s_K)= R(s_K)S'(s_K)
=\frac{[\Re(Z(s_K))]^2}{|Z(s_K)|^2}=(1+O(b^{\frac 14})),
\]
where the last estimate is obtained from \eqref{on:Z} and Lemma~\ref{le:Airy}.

Next, we compute
\begin{align*}
&\frac d{ds} \mathcal W(R,S)
= RS'' - R'' S = \ell \left(\Im (Z)S - \Im(T) R\right)
=\ell \left(\Im(Z)\Re(T)-\Im(T)\Re(Z)\right)\\
&\quad =-\ell\Im(\bar Z T)=-\ell|Z|^2 \Im \left(\int_s^{s_K} \frac{ds'}{Z^2(s')}\right)
=\ell|Z|^2 \Im \left(\int_s^{s_K} \frac{\Im(Z^2(s'))}{|Z(s')|^4} ds'\right)
\end{align*}
By \eqref{on:Z} and \eqref{airy2}, we estimate for $s\in [s_I,s_K]$,
\[
\left|\frac d{ds} \mathcal W(R,S)\right|
\leq C |\ell| \langle s\rangle^{-\frac 12} [\omega(s)]^2 \int_s^{s_K} \langle s'\rangle^{\frac 12} [\omega(s')]^{-2} ds'.\]
By integration, we have for $0\leq s\leq s_K$, 
\[
\int_s^{s_K} \langle s'\rangle^{\frac 12} [\omega(s')]^{-2} ds'
\leq C [\omega(s)]^{-2},
\]
and for $s_I\leq s\leq 0$,
\[
\int_s^{s_K} \langle s'\rangle^{\frac 12} [\omega(s')]^{-2} ds'
\leq \int_s^0+\int_0^{s_K} \leq C \langle s\rangle^{\frac 32}+ C
\leq C \langle s\rangle^{\frac 32}.
\]
Thus, using also \eqref{trois} and \eqref{def:l},
we obtain, for $s\in [s_I,s_K]$ (see~\eqref{on:sI}-\eqref{on:s0}),
\[
\left|\frac d{ds} \mathcal W(R,S)\right|
\leq C \sigma b^{\frac 13} \langle b^{\frac 23}s\rangle^{-\frac 12} \langle s\rangle 
\leq C \sigma \langle s\rangle^{\frac 12}
\leq C b^{-1} \sigma .
\]
Let any $s\in [s_I,s_K]$.
Integrating on $[s,s_K]$ and then using \eqref{on:bb} and \eqref{SM}, it follows that
\[
|\mathcal W(R,S)(s)- \mathcal W(R,S)(s_K)|\leq C b^{-3} \sigma \leq b^{\frac 14}.
\]
This completes the proof of~\eqref{eq:WrRS}.
\end{proof}

We decompose $\Im(Z)$ using $R$ and $S$:
\begin{equation}\label{e:dIz}\left\{\begin{aligned}
\Im(Z) & = \eta R + \mu S \\
\Im(Z') & = \eta R' + \mu S'.
\end{aligned}\right.\end{equation}
\begin{lemma} It holds
\begin{equation}\label{eq:etamu}
\eta(s_K)=O(b^{\frac 14})\quad\mbox{and}\quad \mu(s_K)=-\pi (1+ O(b^{\frac 14})).
\end{equation}
\end{lemma}
\begin{remark}
This result is the key of the proof of the expected special behavior \eqref{Uext1}-\eqref{Uext4}.
Indeed, it relates the properties of the function $Z$ at $s_I$ given by \eqref{on:Z}
to a sharp property at $s_K$, \emph{i.e.} on the other side of the turning point $s=0$.
For this, we exploit the structure of the equation of $Z$ ($k$ is real-valued and $\ell$ is small compared to $\sigma$) in a more refined way compared to the fixed-point argument in Step 1.
\end{remark}
\begin{proof}
First, we estimate $\eta(s_K)$. Using $S(s_K)=0$, \eqref{on:Z} and then Lemma~\ref{le:Airy} (ii), it holds
\begin{align*}
\eta(s_K)&=\frac{\Im(Z(s_K))S'(s_K)-\Im(Z'(s_K)) S(s_K)}{\mathcal W(R(s_K),S(s_K))}\\
&=\frac{\Im(Z(s_K))}{\Re(Z(s_K)) }
=\frac{\Im(\BB(s_K) (1+O(b^{\frac 14}))}{\Re(\BB(s_K)(1+O(b^{\frac 14})) }
=O(b^\frac 14).
\end{align*}

Next, we estimate $\mu(s_I)$, using \eqref{on:Z}, Lemma~\ref{le:Airy} and \eqref{eq:WrRS}
\begin{align*}
\mu(s_I) & = \frac{\Im(Z'(s_I))R(s_I)-\Im(Z(s_I)) R'(s_I)}{\mathcal W(R(s_I),S(s_I))}\\
& = \frac{\Im(Z'(s_I))\Re(Z(s_I))-\Im(Z(s_I)) \Re(Z'(s_I))}{\mathcal W(R(s_I),S(s_I))}\\
& = \frac{\Im \left(\bar Z(s_I)Z'(s_I)\right)}{\mathcal W(R(s_I),S(s_I))}
 = -\pi (1+O(b^{\frac 14})).
\end{align*}

Last, we estimate $\mu(s_K)$ using the equation of $\mu'(s)$.
From \eqref{e:dIz}, \eqref{eq:Z}, \eqref{eq:R} and \eqref{eq:S}, we have
\begin{equation}\label{sys:detadmu}\left\{\begin{aligned}
\eta' R + \mu' S &= 0 \\
 \eta' R'+ \mu' S' &= \ell \left(\eta \Im(Z)+ \mu \Im(T)+R \right).
\end{aligned}\right.\end{equation}
We claim the following identity
\begin{equation}\label{biz}
 \eta \Im(Z)+ \mu \Im(T)+R 
=\frac{R}{\mathcal W(R,S)}.
\end{equation}
Indeed, from \eqref{e:dIz}, we have first,
\[
\eta= \frac{\Im(Z) S'-\Im(Z') S}{\mathcal W(R,S)},\quad
\mu=\frac{\Im(Z') R-\Im(Z) R'}{\mathcal W(R,S)}.
\]
Thus,
\begin{align*}
\mathcal W(R,S)\left( \eta \Im(Z)+ \mu \Im(T)+R \right)
&= \Im(Z)\Im(Z)\Re(T')-\Im(Z)\Im(Z')\Re(T)\\
&\quad +\Im(Z')\Re(Z)\Im(T)-\Im(Z)\Re(Z')\Im(T)\\
&\quad+\Re(Z)\Re(Z)\Re(T')-\Re(Z)\Re(Z')\Re(T)\\
& = |Z|^2 \Re(T')-\Re(\bar Z Z')\Re(T)+\Im(\bar ZZ')\Im(T)\\
&= \Re\left( |Z|^2 T' -\bar Z Z'T \right).
\end{align*}
But multiplying $\mathcal W(Z,T)=1$ by $\bar Z$, we deduce 
the identity
$|Z|^2 T'-\bar Z Z'T= \bar Z$, so that $ \Re\left( |Z|^2 T' -\bar Z Z'T \right)=\Re(\bar Z)=R$.
Thus, \eqref{biz} is proved.

From \eqref{sys:detadmu} and \eqref{biz}, we have obtained
\[
\mu'= \frac{\ell R^2}{[\mathcal W(R,S)]^2}.\]
We deduce from \eqref{on:R}, \eqref{eq:WrRS} and \eqref{airy2} that for any $s\in [s_I,s_K]$,
\begin{equation*}
|\mu'(s)|\leq C |\ell| \langle s\rangle^{-\frac 12} [\omega(s)]^2.
\end{equation*}
Using \eqref{pour2}, 
\[
\int_0^{s_K} \langle s\rangle^{-\frac 12} [\omega(s)]^2ds 
\leq Cb^{-\frac 13} [\omega(s_K)]^2\]
and using $\omega(s')=1$ for $s'\leq 0$, 
\[
\int_{s_I}^{0} \langle s\rangle^{-\frac 12} [\omega(s)]^2ds\leq C b^{-1}.
\]
Thus, using $|\ell|\leq C \sigma $ from the definition of $\ell$ in \eqref{def:l} and \eqref{trois}, next using ~\eqref{pourSk}, \eqref{on:bb}, \eqref{SM} and we obtain
\[
|\mu(s_I)-\mu(s_K)|\leq C \sigma [w(s_K)]^2 \leq C b^{-1}\exp\left(-\frac 1{\sqrt{b}}\right)\leq b^{\frac 14}.
\]
This completes the proof of \eqref{eq:etamu}.
\end{proof}

\begin{lemma}
There exists $\theta_K$ with 
$\theta_K=O(b^{\frac 14})$ such that the function 
\[
\tilde Z(s)=\ee^{\ii \theta_K} Z(s)
\]
satisfies
\begin{align}
\Re (\tilde Z(s_K)) &= \Re(\BB (s_K)) (1+O(b^\frac 14)), \label{eq:RetZ}\\
\Re (\tilde Z'(s_K)) &= \Re(\BB' (s_K)) (1+O(b^\frac 14)),\label{eq:RedtZ}\\
\Im(\tilde Z(s_K))&=\pi \Ai(s_K) ,\label{eq:ImndtZ}\\
\Im(\tilde Z'(s_K))&=\pi \Ai'(s_K) (1+O(b^\frac 14)).\label{eq:ImtZ}
\end{align}
\end{lemma}
\begin{remark}
This result means that the above solution $\tilde Z$ of equation \eqref{eq:Z} mimics the behavior~\eqref{ImB}
of the function $\BB$ at the main order. As explained before, this result is directly related to the fact that the function $k$ in \eqref{eq:Z} is real-valued and that the smallness of the function $\ell$ is related to $\sigma$.
\end{remark}
\begin{proof}

We claim that there exists $\theta_K=O(b^{\frac 14})$ such that
\begin{equation}\label{eq:thetaK}
\sin \theta_K + \eta(s_K) \cos \theta_K = \pi \frac{\Ai(s_K)}{\Re(Z(s_K))}.
\end{equation}
Indeed, we first set $\tilde \theta_K = -\arctan (\eta(s_K))$, so that the equation on
$\theta_K$ becomes
\[
\sin(\theta_K-\tilde \theta_K)
=\sin(\theta_K) \cos (\tilde\theta_K) - \cos (\theta_K)\sin(\tilde\theta_K) 
=\pi \cos(\tilde\theta_K) \frac{\Ai(s_K)}{\Re(Z(s_K))}.
\]
Note that $\tilde \theta_K=O(b^\frac 14)$ by \eqref{eq:etamu}. Note also from \eqref{on:R}, Lemma~\ref{le:Airy}
and then \eqref{on:sK}-\eqref{on:s0}, that (for $b$ small enough)
\[\left|\frac{\Ai(s_K)}{\Re(Z(s_K))}\right|\leq C \exp\left( -\frac 43 s_K^{\frac 32}\right)
\leq C \exp\left( -\frac{\pi}{2b}\right)
\leq b^{\frac 14}.\]
Thus, $\theta_K$ defined by
\[
\theta_K=\tilde \theta_K + \arcsin\left(\pi \cos(\tilde\theta_K) \frac{\Ai(s_K)}{\Re(Z(s_K))} \right)
\]
is solution of \eqref{eq:thetaK} and satisfies $\theta_K=O(b^{\frac 14})$.
We set $\tilde Z(s)=\ee^{\ii \theta_K} Z(s)$.

Proof of \eqref{eq:RetZ}-\eqref{eq:RedtZ}.
By \eqref{e:dIz}, $S(s_K)=0$, \eqref{eq:etamu}, $\theta_K=O(b^{\frac 14})$ and \eqref{on:R}, it holds
\begin{align*}
\Re(\tilde Z(s_K))
&=\Re(Z(s_K)) \cos \theta_K-\Im(Z(s_K)) \sin \theta_K\\
&=\left(\cos \theta_K - \eta(s_K)\sin \theta_K\right) \Re(Z(s_K))=
\Re(\BB(s_K))(1+O(b^{\frac 14})) .
\end{align*}
Similarly, by \eqref{e:dIz}, \eqref{eq:etamu}, $\theta_K=O(b^{\frac 14})$, \eqref{on:R}
and (from \eqref{on:Z} and Lemma~\ref{le:Airy})
\[S'(s_K)=\frac{\Re(Z(s_K))}{|Z(s_K)|^2}=\frac 1{\Re(\BB(s_K))}(1+O(b^\frac 14)),\]
it holds
\begin{align*}
\Re(\tilde Z'(s_K))
&=\Re(Z'(s_K)) \cos \theta_K-\Im(Z'(s_K)) \sin \theta_K\\
&=\left(\cos \theta_K - \eta(s_K)\sin \theta_K\right) \Re(Z'(s_K))
-\mu(s_K)\sin (\theta_K )S'(s_K)\\
& = \Re(\BB'(s_K))(1+O(b^{\frac 14})) .
\end{align*}

Proof of \eqref{eq:ImndtZ}-\eqref{eq:ImtZ}. By \eqref{e:dIz}, $S(s_K)=0$ and \eqref{eq:thetaK}, it holds
\begin{align*}
\Im(\tilde Z(s_K))
&=\Re(Z(s_K)) \sin \theta_K+\Im(Z(s_K)) \cos \theta_K\\
&=\left(\sin \theta_K + \eta(s_K)\cos \theta_K\right) \Re(Z(s_K))=\pi \Ai(s_K). 
\end{align*}
By \eqref{e:dIz}, \eqref{eq:etamu} and \eqref{eq:thetaK}, it also holds
\begin{align*}
\Im(\tilde Z'(s_K))
&=\Re(Z'(s_K)) \sin \theta_K+\Im(Z'(s_K)) \cos \theta_K\\
&=\left(\sin \theta_K + \eta(s_K)\cos \theta_K\right) \Re(Z'(s_K))
+\mu(s_K)\cos (\theta_K) S'(s_K)\\
& = \pi \Ai(s_K)\frac{\Re(Z'(s_K))}{\Re(Z(s_K))}
-\pi S'(s_K) (1+ O(b^{\frac 14})).
\end{align*}
Note that by \eqref{on:Z}, \eqref{on:R} and Lemma~\ref{le:Airy}
\begin{align*}
&\Ai(s_K) \frac {R'(s_K)}{R(s_K)} = \Ai(s_K)\Re\left(\frac {\BB'(s_K)}{\BB(s_K)}\right)(1+O(b^{\frac 14}))
=-\Ai'(s_K) (1+O(b^{\frac 14})),\\
&S'(s_K)=\frac{\Re(Z(s_K))}{|Z(s_K)|^2}=\frac 1{\Re(\BB(s_K))}(1+O(b^\frac 14))=-2\Ai'(s_K) (1+O(b^\frac 14)).
\end{align*}
Thus,
$\Im(\tilde Z'(s_K))=\pi \Ai'(s_K) (1+O(b^\frac 14))$.
\end{proof}
Now, we set
\begin{equation}\label{def:thetabis}
U_\ext(r)= \ee^{\ii \theta_\ext} U(r)\quad \mbox{where}\quad
\ee^{\ii \theta_\ext}=\frac{\bar\beta_I}{|\beta_I|}\ee^{\ii \theta_K}, \quad \theta_\ext\in [0,2\pi),
\end{equation}
 and we prove
that estimates \eqref{Uext1}-\eqref{Uext4} for $U_\ext$ follow from \eqref{eq:RetZ}-\eqref{eq:ImtZ}.
Indeed, by the definition of $\theta_\ext$ in \eqref{def:thetabis} and the definition of $Z$ in \eqref{def:Z}, we   observe that
\begin{align*}
|\beta_I| \tilde Z(s_K) & =|\beta_I| \ee^{\ii \theta_K} Z(s_K)
=\beta_I \ee^{\ii \theta_\ext} Z(s_K) = \ee^{\ii \theta_\ext} Y(s_K),\\
|\beta_I| \tilde Z'(s_K) & = \ee^{\ii \theta_\ext} Y'(s_K).
\end{align*}
Therefore, using the definition of $U_\ext$ in \eqref{def:thetabis}, similarly as in \eqref{yI} and \eqref{dyI}, we have
\begin{align}
|\beta_I|\tilde Z(s_K)&= [\zeta'(\tau_K)]^{\frac12}{U_\ext(r_K)},\label{ZU1}\\
|\beta_I|\tilde Z'(s_K)&
=-b^{-\frac 13}[\zeta'(\tau_K )]^{-\frac 12} U_\ext'(r_K)
+\frac 12 b^{\frac 23}\zeta''(\tau_K)[\zeta'(\tau_K)]^{-\frac 32} U_\ext(r_K).\label{ZU2}
\end{align}
First, from \eqref{ZU1}, \eqref{eq:aIbI}, \eqref{eq:dzd2z} next \eqref{eq:RetZ}-\eqref{eq:ImtZ}
and Lemma~\ref{le:Airy}, last using~\eqref{ee:sKbis}, we obtain
\begin{align*}
\Re(U_\ext(r_K))&=
\frac {\rho b^{\frac 13}}{\sqrt{2\pi}} \left(\frac{3\pi}4\right)^{\frac 16} \Re(\tilde Z(s_K))(1+O(b^\frac 14))\\
&= \frac{\rho b^{\frac 13}}{\sqrt{2\pi}}\left(\frac{3\pi}{4}\right)^{\frac 16}\Re(\BB(s_K))(1+O(b^\frac 14))\\
&=\frac{\rho b^{\frac 13}}{\sqrt{2}} \left(\frac{3\pi}{4}\right)^{\frac 16} s_K^{-\frac 14}\exp\left(\frac 23 s_K^{\frac 32}\right)(1+O(b^\frac 14))\\
&= \frac{\rho \sqrt{b}}{\sqrt{2}}
\exp\left(\frac{\pi}{2b}\right) \exp\left(-\frac1{\sqrt{b}}\right) (1+O(b^{\frac 14})).
\end{align*}
and
\begin{align*}
\Im(U_\ext(r_K))&=
\frac {\rho b^{\frac 13}}{\sqrt{2\pi}} \left(\frac{3\pi}4\right)^{\frac 16} \Im(\tilde Z(s_K))(1+O(b^\frac 14))\\
&= \frac{\rho b^{\frac 13}}{\sqrt{2\pi}}\left(\frac{3\pi}{4}\right)^{\frac 16}\pi\Ai(s_K)(1+O(b^\frac 14))\\
&= \frac{\rho b^{\frac 13}}{2\sqrt{2}} \left(\frac{3\pi}{4}\right)^{\frac 16}s_K^{-\frac 14}\exp\left(-\frac 23 s_K^{\frac 32}\right)(1+O(b^\frac 14))\\
&= \frac{\rho \sqrt{b}}{2\sqrt{2}}
\exp\left(-\frac{\pi}{2b}\right) \exp\left(\frac1{\sqrt{b}}\right) (1+O(b^{\frac 14})).
\end{align*}
Second, from \eqref{eq:aIbI}, next \eqref{eq:RetZ}-\eqref{eq:ImtZ}
and Lemma~\ref{le:Airy}, last \eqref{ee:sKbis}, we have
\begin{align*}
-|\beta_I|b^{\frac 13}[\zeta'(\tau_K )]^{\frac 12} \Re(\tilde Z'(s_K))
&=-\frac{\rho b^{\frac 23}}{\sqrt{2\pi}}\left(\frac{3\pi}{4}\right)^{-\frac 16}
\Re(\BB'(s_K))(1+O(b^\frac 14))\\
&=-\frac{\rho b^{\frac 23}}{\sqrt{2 }} \left(\frac{3\pi}{4}\right)^{-\frac 16}s_K^{\frac 14}\exp\left(\frac 23 s_K^{\frac 32}\right)(1+O(b^\frac 14))\\
&=-\frac{\rho \sqrt{b}}{\sqrt{2}}\exp\left(\frac{\pi}{2b}\right) \exp\left(-\frac1{\sqrt{b}}\right) (1+O(b^{\frac 14})).
\end{align*}
Since 
\[
b \Re(U_\ext(r_K))=O\left( \rho b^{\frac 32} \exp\left(\frac{\pi}{2b}\right) \exp\left(-\frac1{\sqrt{b}}\right)\right),
\]
we obtain from \eqref{ZU2},
\[
\Re(U_\ext'(r_K))=
-\frac{\rho \sqrt{b}}{\sqrt{2}}\exp\left(\frac{\pi}{2b}\right) \exp\left(-\frac1{\sqrt{b}}\right) (1+O(b^{\frac 14})).
\]
Last, using the same estimates as before,
\begin{align*}
-|\beta_I|b^{\frac 13}[\zeta'(\tau_K )]^{\frac 12} \Im(\tilde Z'(s_K))
&=-\frac{\rho b^{\frac 23}}{\sqrt{2\pi}} \left(\frac{3\pi}{4}\right)^{-\frac 16}
\pi \Ai'(s_K)(1+O(b^\frac 14))\\
&=\frac{\rho b^{\frac 23}}{2\sqrt{2}} \left(\frac{3\pi}{4}\right)^{-\frac 16}s_K^{\frac 14}\exp\left(-\frac 23 s_K^{\frac 32}\right)(1+O(b^\frac 14))\\
&=\frac{\rho \sqrt{b}}{2\sqrt{2}}\exp\left(-\frac{\pi}{2b}\right) \exp\left(\frac1{\sqrt{b}}\right) (1+O(b^{\frac 14})).
\end{align*}
Since 
\[
b \Im(U_\ext(r_K))=O\left( \rho b^{\frac 32} \exp\left(-\frac{\pi}{2b}\right) \exp\left(\frac1{\sqrt{b}}\right)\right),
\]
we obtain from \eqref{ZU2},
\[
\Im(U_\ext'(r_K))=
\frac{\rho \sqrt{b}}{2\sqrt{2}}\exp\left(-\frac{\pi}{2b}\right) \exp\left(\frac1{\sqrt{b}}\right) (1+O(b^{\frac 14})).
\]
We have proved \eqref{Uext1}-\eqref{Uext4}.
Finally, note that \eqref{Pext1}-\eqref{Pext4} then follow immediately from the definition of $P_\ext$
in~\eqref{def:Pext}.
\end{proof}

\section{Solutions of the nonlinear equation on $K$}
The objective of this section is to construct a family of solutions $P_\itt$ of \eqref{eq:Pb} on the interval $K$ with sufficiently many free parameters to perform later the matching at $r=r_K$ with the solutions $P_\ext$ of \eqref{eq:Pb} on the interval $J\cup I$ constructed in Proposition~\ref{pr:Uext}.
It is also important to obtain precise information on the general form of these solutions at the matching point $r=r_K$ that coincide with \eqref{Pext1}-\eqref{Pext4}.
The expected solutions $P_\itt$ are close to the ground state solitary wave $Q$, and
we will construct them by fixed-point using the properties of the linearized operator around $Q$.

\subsection{Preliminaries on linearized operators}
Let
\begin{align}
\Lp&= -\partial_{rr} -\frac{d-1}r \partial_r +1 - pQ^{p-1}, \label{def:Lp}\\
\Lm&= -\partial_{rr} -\frac{d-1}r \partial_r +1 - Q^{p-1}. \label{def:Lm}
\end{align}
We gather in the next lemmas standard information concerning $\Lp$ and $\Lm$.
\begin{lemma}\label{le:7}
There exist $\mathcal C^2$ functions $A:[0,\infty)\to \R$ and $D:(0,+\infty)\to \R$ that are independent solutions of the equation $\Lp y=0$ on $(0,\infty)$ and satisfy the following properties.
\begin{enumerate}
\item The function $A$ satisfies $A(0)=1$, $A'(0)=0$ and for a constant $\kappa_A\neq 0$
\begin{align*}
 A(r)&= \kappa_A r^{-\frac{d-1}2} \ee^{r} \left(1+O(r^{-1})\right)\quad \mbox{on $[1,\infty)$,}\\
 A'(r)&= \kappa_A r^{-\frac{d-1}2} \ee^{r} \left(1+O(r^{-1})\right)\quad \mbox{on $[1,\infty)$.}
\end{align*}
\item The function $D$ satisfies
\begin{align*}
D(r)&= \begin{cases}
(2-d)^{-1} r^{2-d}\left(1+O(r)\right) & \mbox{for $d\geq 3$}\\
\log r \left(1+O(r)\right) & \mbox{for $d=2$}\\
r \left(1+O(r)\right) & \mbox{for $d=1$}
\end{cases}\quad \mbox{on $(0,1]$,}\\
D'(r)&= r^{-d+1}(1+O(r))\quad \mbox{on $(0,1]$,}
\end{align*}
and
\begin{align*}
D(r)&= - (2\kappa_A)^{-1} r^{-\frac{d-1}2} \ee^{-r}\left(1+O(r^{-1})\right) \quad \mbox{on $[1,\infty)$,}\\
D'(r)&= (2\kappa_A)^{-1} r^{-\frac{d-1}2} \ee^{-r}\left(1+O(r^{-1})\right) \quad \mbox{on $[1,\infty)$.}
\end{align*}
\item Moreover, for all $r>0$,
\begin{equation*}
\mathcal W(A,D)=A D'-A'D= 
r^{-d+1}.
\end{equation*}
\end{enumerate}
\end{lemma}
\begin{proof}
(i) We consider the equation $\Lp A=0$ with initial condition at $r=0$, $A(0)=1$ and $A'(0)=0$.
This problem rewrites under the following integral form
\[
A(r)=1+\int_0^r s^{-(d-1)} \int_0^s \tau^{d-1} (A-pQ^{p-1} A)(\tau) d\tau,\quad r\geq 0.
\]
This equation is solved locally around $r=0$ by a usual fixed point method. For $r>0$, the equation is not anymore singular
and standard methods apply. In particular, the local solution extends globally for $r\geq 0$ since the equation is linear.
By standard ODE techniques, if $A$ converges to zero as $r\to\infty$, then it decays exponentially. This is ruled out by
Lemma~2.1 of~\cite{CGNT} which describes the spectrum of the operator $-\Delta+1-pQ^{p-1}$ on $L^2(\R^N)$.
The exact behavior of $A$ as $r\to \infty$ then follows from standard ODE techniques.

(ii)-(iii) We define $D$ on $(0,\infty)$ by
\begin{equation*}
D(r)=\begin{cases} -A(r)\int_{r}^\infty \left(s^{-d+1}/A(s)\right) ds & \mbox{for $r>0$ such that $A(r)\neq 0$,}\\
- r^{-d+1}/A'(r)& \mbox{otherwise.}\end{cases}
\end{equation*}
It is easily checked that $D$ is a solution of $\Lp y=0$ on $(0,\infty)$ that also satisfies
$\mathcal W(A,D)(1)=1$.
The asymptotic behavior of $D$ and $D'$ as $r\to \infty$ and as $r\to 0$ are direct consequences of the definition of $D$.
\end{proof}

Next, we define the following norms
\begin{align*}
\Np(f)&= \left\| f/Q\right\|_{L^\infty(K)},\\
\Nm(f)&= \left\| H f\right\|_{L^\infty(K)},\quad H(r)=(1+r)^{d-1} Q(r).
\end{align*}
For future use, we remark that since $r_K=b^{-\frac 12}$, by \eqref{on:Q},
\begin{equation}\label{fut1}
\Np(f)
\leq \|1/(QH)\|_{L^\infty(K)} \, \Nm(f)
\leq C \ee^{\frac 2{\sqrt{b}}} \Nm(f).
\end{equation}

\begin{lemma}\label{le:8}
For any continuous function $f:[0,+\infty)\to \R$, let
\begin{equation*}
[\Hp(f)](r)=
- \left\{ A(r) \int_r^{r_K} f(s) D(s) s^{d-1} ds 
+ D(r) \int_0^r f(s) A(s) s^{d-1} ds\right\}.
\end{equation*}
Then, it holds
\begin{equation*}
\Lp (\Hp(f))=f.
\end{equation*}
Moreover,
\begin{align}
&\Np(\Hp(f))\leq C b^{-\frac 12} \Np(f),\quad
\Np(\Hp(f))\leq C \Np((1+r^2)f),\label{e:Np}\\
&\Np([\Hp(f)]')\leq C b^{-\frac 12} \Np(f),\quad
\Np([\Hp(f)]')\leq C \Np((1+r^2)f),\label{e:DNp}
\end{align}
\end{lemma}
\begin{proof}
The equation $\Lp (\Hp(f))=f$ is obtained by direct computations from the definition of $\Hp(f)$ and the properties of 
$A$ and $D$ : $\Lp A=\Lp D=0$ and $AD'-A'D=r^{-d+1}$.

By the definition of $\Np(f)$, \eqref{on:Q}, and the bounds on $A$ and $D$ from Lemma~\ref{le:7}, we have
\begin{align*}
\left|\frac{A(r)}{Q(r)} \int_r^{r_K} f(s) D(s) s^{d-1} ds \right|
&\leq C \Np(f) \frac{|A(r)|}{Q(r)} \int_{r}^{r_K} Q(s)|D(s)| s^{d-1} ds\\
&\leq C\Np(f) \ee^{2r} \int_{r}^{r_K} \ee^{-2s} ds\leq C \Np(f),\end{align*}
 for any $0<r<r_K$. Similarly,
\begin{align*}
\left| \frac{D(r)}{Q(r)} \int_0^{r} f(s) A(s) s^{d-1} ds \right|
&\leq C \Np(f)\frac{|D(r)|}{Q(r)}\int_{0}^{r} Q(s)|A(s)| s^{d-1} ds\\
&\leq C r_K\Np(f)\leq C b^{-\frac 12} \Np(f),\end{align*}
and
\begin{align*}
&\left| \frac{D(r)}{Q(r)} \int_0^{r} f(s) A(s) s^{d-1} ds \right|\\
&\leq C \Np((1+r^2)f) \frac{|D(r)|}{Q(r)} \int_{0}^{r} \frac{Q(s)|A(s)|}{1+s^2} s^{d-1} ds
\leq C \Np((1+r^2)f),\end{align*}
for any $0<r<r_K$. Those estimates prove~\eqref{e:Np}.

Next, we observe that
\begin{equation*}
[\Hp(f)]'(r)=
- \left\{ A'(r) \int_r^{r_K} f(s) D(s) s^{d-1} ds 
+ D'(r) \int_0^r f(s) A(s) s^{d-1} ds\right\},
\end{equation*}
and similar estimates yield~\eqref{e:DNp}.
In particular, note that
\begin{align*}
\left| Q^{-1}(r) D'(r) \int_0^{r} f(s) A(s) s^{d-1} ds \right|
&\leq C \Np(f) \frac{| D'(r)|}{Q(r)}\int_{0}^{r} Q(s)|A(s)| s^{d-1} ds\\
&\leq C r_K\Np(f)\leq C b^{-\frac 12} \Np(f),\end{align*}
where for $d\geq 2$, the stronger singularity of the function $D'(r)$ at $r=0$ is compensated by the multiplication by
$\int_{0}^{r} Q(s)|A(s)| s^{d-1} ds\leq C r^d$.
\end{proof}

\begin{lemma}
For any continuous function $f:[0,+\infty)\to \R$, let
\begin{equation*}
[\Hm(f)](r)=
- Q(r) \int_0^r \left\{\int_0^s f(\tau) Q(\tau) \tau^{d-1} d\tau\right\} \frac{ds}{Q^2(s) s^{d-1}}.
\end{equation*}
Then, it holds
\begin{equation*}
\Lm (\Hm(f))=f.
\end{equation*}
Moreover,
\begin{align}
&\Nm(\Hm(f))\leq C b^{-\frac 12} \Nm(f),\quad
\Nm(\Hm(f))\leq C \Nm((1+r^2)f)\leq C \Np(f),\label{e:Nm}\\
&\Nm([\Hm(f)]')\leq C b^{-\frac 12} \Nm(f),\quad
\Nm([\Hm(f)]')\leq C \Nm((1+r^2)f)\leq C \Np(f).\label{e:DNm}
\end{align}
\end{lemma}
\begin{proof}
Note that the equation of $Q$ rewrites $\Lm Q=0$. By the definition of $\Hm(f)$ and direct computation, we obtain
the relation $\Lm (\Hm(f))=f$.

By the definition of $\Nm(f)$, we have
\begin{align*}
\left| H(r) \Hm(f) \right|
&\leq C \Nm(f) \ee^{-2r} \int_0^r\left\{\int_0^s \frac{Q(\tau)}{H(\tau)} \tau^{d-1} d\tau\right\} \ee^{2s}\left(\frac{1+s}{s}\right)^{d-1}ds \\
&\leq C\Nm(f) \ee^{-2r} \int_0^r \left\{\int_0^s \left(\frac{\tau}{1+\tau}\right)^{d-1} d\tau\right\} \ee^{2s} \left(\frac{1+s}{s}\right)^{d-1} ds
\\ &\leq C\Nm(f) \ee^{-2r} \int_0^r s \ee^{2s} ds \leq C b^{-\frac 12} \Nm(f) ,\end{align*}
 for any $0<r<r_K$. Moreover,
 \begin{align*}
&\left| H(r) \Hm(f) \right|\\
&\quad \leq C \Nm((1+r^2)f) \ee^{-2r} \int_0^r\left\{\int_0^s \frac{Q(\tau)}{H(\tau)}\frac{\tau^{d-1}}{1+\tau^2} d\tau\right\} \ee^{2s}\left(\frac{1+s}{s}\right)^{d-1}ds \\
&\quad \leq C\Nm((1+r^2)f) \ee^{-2r} \int_0^r \ee^{2s} ds \leq C \Nm((1+r^2)f) \leq C \Np(f),\end{align*}
 for any $0<r<r_K$, which proves~\eqref{e:Nm}.
 
 Next, we have
 \begin{align*}
[\Hm(f)]'(r)&=
- Q'(r) \int_0^r \left\{\int_0^s f(\tau) Q(\tau) \tau^{d-1} d\tau\right\} \frac{ds}{Q^2(s) s^{d-1}}\\
&\quad -\frac{1}{Q(r) r^{d-1}}\int_0^r f(\tau) Q(\tau) \tau^{d-1} d\tau,
\end{align*}
where the first term on the right-hand side is estimated exactly as $\Hm(f)$.
For the second term, we proceed as follows
\begin{align*}
\left| \frac{H(r)}{Q(r) r^{d-1}}\int_0^r f(\tau) Q(\tau) \tau^{d-1} d\tau \right|
&\leq C \Nm(f) \left(\frac{1+r}r\right)^{d-1} \int_0^r \frac{Q(\tau)}{H(\tau) } \tau^{d-1} d\tau \\
&\leq C b^{-\frac 12} \Nm(f),
\end{align*}
and similarly,
\begin{equation*}
\left| \frac{H(r)}{Q(r) r^{d-1}}\int_0^r f(\tau) Q(\tau) \tau^{d-1} d\tau \right|
\leq C \Nm((1+r^2)f),
\end{equation*}
for any $0<r<r_K$, which implies~\eqref{e:DNm}. 
\end{proof}

\begin{lemma}\label{le:10}
There exists a $\mathcal C^2$ solution $B:[0,\infty)\to \R$ of the equation $\Lm y=-Q$ on $(0,\infty)$ satisfying the following properties: $B(0)=B'(0)=0$ and
\begin{align*}
 B(r)&= \kappa_B r^{-\frac{d-1}2} \ee^{r} \left(1+O(r^{-1})\right)\quad \mbox{on $[1,\infty)$,}\\
 B'(r)&= \kappa_B r^{-\frac{d-1}2} \ee^{r} \left(1+O(r^{-1})\right)\quad \mbox{on $[1,\infty)$,}
\end{align*}
where
\begin{equation}\label{eq:kappaB}
\kappa_B=\frac1{2\kappa}{\int_0^\infty Q^2(r) r^{d-1} dr} =\frac{N_c}{2\kappa}>0.
\end{equation}
\end{lemma}
\begin{proof}
We define
\begin{equation*}
B(r)=- [\Hm (Q)](r)= Q(r) \int_0^r \left\{\int_0^s Q^2(\tau) \tau^{d-1} d\tau\right\} \frac{ds}{Q^2(s) s^{d-1}}
\end{equation*}
so that
\begin{align*}
B'(r)&= Q'(r) \int_0^r \left\{\int_0^s Q^2(\tau) \tau^{d-1} d\tau\right\} \frac{ds}{Q^2(s) s^{d-1}}
\\ &\quad + \frac{1}{Q(r) r^{d-1}} \int_0^r Q^2(s) s^{d-1} ds .
\end{align*}
From Lemma~\ref{le:8}, we have $\Lm B = -Q$.
The values $B(0)=0$ and $B'(0)=0$ follow from the expressions of $B$ and $B'$.

Next, we have the bound
\begin{equation*}
\int_0^r \left\{\int_0^s Q^2(\tau) \tau^{d-1} d\tau\right\} \frac{ds}{Q^2(s) s^{d-1}}
\leq C \int_{0}^r \ee^{2s} ds\leq C \ee^{2r},
\end{equation*}
which implies $|B(r)|+|B'(r)|\leq C r^{-\frac{d-1}2} \ee^{r}$.
Note that
\begin{equation*}
\left| \int_0^r \left\{\int_s^\infty Q^2(\tau) \tau^{d-1} d\tau\right\} \frac{ds}{Q^2(s) s^{d-1}}\right|
\leq C (r+1),
\end{equation*}
and thus using the definition of $\kappa_B$ in \eqref{eq:kappaB},
\[
\left| B(r) - 2\kappa \kappa_B Q \int_0^r \frac{ds}{Q^2(s)s^{d-1}}\right|
\leq C (r+1) Q(r).
\]
Moreover, by \eqref{on:Q}
\begin{align*}
\int_0^r \left|\frac 1{Q^2(s)s^{d-1}}-\kappa ^{-2} \ee^{2s}\right| ds
\leq C \int_{0}^r (1+s)^{-1} \ee^{2s} ds\leq C(1+r)^{-1} \ee^{2r}.
\end{align*}
The estimate
\begin{equation*}
\left| B(r) -\kappa_B r^{-\frac{d-1}2} \ee^{r}\right|\leq C r^{-\frac{d+1}2} \ee^{r},
\end{equation*}
and a similar estimate for $B'$ then follow using \eqref{on:Q} again.
\end{proof}

\subsection{Construction of a family of solutions on $K$}
We construct a family of solutions of \eqref{eq:Pb} on $K$ close to the ground state solitary wave
$Q$. 

\begin{proposition}\label{pr:Pint}
For $\sigma>0$ small enough and
for any $b$, $\gamma$ satisfying \eqref{on:b} and~\eqref{on:gamma},
there exists a solution $P_\itt=P_\itt[\sigma,b,\gamma]$ of~\eqref{eq:Pb} on $K$ satisfying
\begin{align}
\Re(P_\itt(r_K))&= \kappa b^{\frac{d-1}4} \exp\left(-\frac1{\sqrt{b}}\right) (1+O(b^{\frac 13}))
+\kappa_A\gamma b^{\frac {d-1}4} \exp\left(\frac1{\sqrt{b}}\right) ,\label{Pint1}\\
\Re(P_\itt'(r_K))&= -\kappa b^{\frac{d-1}4} \exp\left(-\frac1{\sqrt{b}}\right) (1+O(b^{\frac 13}))
+ \kappa_A \gamma b^{\frac {d-1}4} \exp\left(\frac1{\sqrt{b}}\right) ,\label{Pint2}\\
\Im(P_\itt(r_K))&= \kappa_B \sigma b^{1+\frac{d-1}4} \exp\left(\frac1{\sqrt{b}}\right) (1+O(b^{\frac 14}))\label{Pint3} ,\\
\Im(P_\itt'(r_K))&= \kappa_B \sigma b^{1+\frac{d-1}4} \exp\left(\frac1{\sqrt{b}}\right) (1+O(b^{\frac 14})) ,\label{Pint4}
\end{align}
and 
\begin{equation}\label{eq:closeint}
\|P_\itt-Q\|_{\dot H^1(K)}\leq C b^{\frac 1{12}}.
\end{equation}
Moreover, the map
$
(\sigma,b,\gamma)
\mapsto
(P_\itt[\sigma,b,\gamma](r_K),P_\itt[\sigma,b,\gamma]'(r_K))
$
is continuous.
\end{proposition}
\begin{remark}\label{rk:match}
To prove Proposition~\ref{pr:Pint},  we directly work on the equation~\eqref{eq:Pb}.
The function $B$ introduced in Lemma~\ref{le:10} appears naturally in the imaginary part of the constructed solution $P_\itt$ to absorb at the main order the term $-\ii b\sigma Q$ present in the equation after the linearization $P=Q+\mbox{small}$ (see the exact ansatz in~\eqref{def:P}).
Note that the special behavior of the imaginary part of $P_\ext$
described in \eqref{Pext3}-\eqref{Pext4} will exactly coincide with the exponential growth of $B$ and prescribe the  free parameter~$\sigma$.

Concerning  real parts, 
we observe that the special behavior of the real part of
$P_\ext$ in \eqref{Pext1}-\eqref{Pext2} will correspond at the main order to the  ground state $Q$ itself and will prescribe the choice of the small parameter $\rho$.
Moreover, even if the function~$A$ introduced in Lemma~\ref{le:7}   has an exponentially growing  behavior, it is essential to introduce it in the construction, in relation with the  additional parameter $\gamma$
to have enough free parameters for the matching. The growth of the function $A$ will be compensated by the smallness of~$\gamma$.
\end{remark}
\begin{proof}
We look for a solution $P$ of \eqref{eq:Pb} on the interval $K$ of the form
\begin{equation}\label{def:P}
P= (Q+\gamma A +\phi_+ ) + \ii ( b\sigma B + \phi_- ),
\end{equation}
where 
$\phi_+$, $\phi_-$ are small (in some sense) continuous real-valued functions on $K$ to be determined by a fixed point argument.

The equation \eqref{eq:Pb} for $P$ rewrites
\begin{align*}
0
&= Q'' + \frac{d-1}r Q' + \left( \frac{b^2r^2}4-1-\ii b\sigma\right) Q +Q^p\\
&\quad+ \gamma \left\{A'' + \frac{d-1}r A' + \left( \frac{b^2r^2}4-1-\ii b\sigma\right) A + p Q^{p-1} A\right\}\\
&\quad+ \phi_+'' + \frac{d-1}r \phi_+' + \left( \frac{b^2r^2}4-1-\ii b\sigma\right) \phi_+ + p Q^{p-1} \phi_+\\
&\quad+ \ii b\sigma\left\{B'' + \frac{d-1}r B' + \left( \frac{b^2r^2}4-1-\ii b\sigma\right) B + Q^{p-1} B\right\}\\
&\quad+ \ii\left\{\phi_-'' + \frac{d-1}r \phi_-' + \left( \frac{b^2r^2}4-1-\ii b\sigma\right) \phi_- + Q^{p-1} \phi_-\right\}\\
&\quad + \NLp + \ii\, \NLm,\end{align*}
where $\NLp$ and $\NLm$ are the real and imaginary parts of the nonlinear error term
\begin{equation*}
\NLp + \ii\, \NLm=
|P|^{p-1} P -\left\{Q^p+ p Q^{p-1} (\gamma A +\phi_+ ) 
+ \ii Q^{p-1} ( b\sigma B + \phi_- )\right\}.
\end{equation*}
Using the equation of $Q$ and $\Lp A=0$, $\Lm B=-Q$, 
we deduce that $P$ satisfies \eqref{eq:Pb} if and only if
\begin{align*}
0& = \frac{b^2r^2}4 Q + \left(\frac{b^2r^2}4-\ii b\sigma\right) (\gamma A +\phi_+)-\Lp \phi_+ \\
&\quad+ \ii\left\{ \left( \frac{b^2r^2}4-\ii b\sigma\right) (b\sigma B +\phi_-) -\Lm \phi_- \right\}
+ \NLp + \ii\, \NLm.\end{align*}
Therefore, setting
\begin{align*}
\FFp(\phi_+,\phi_-)&=\frac{b^2r^2}4 ( Q+\gamma A+\phi_+ ) + b\sigma (b\sigma B +\phi_-) +\NLp,\\
\FFm(\phi_+,\phi_-)&=- b\sigma (\gamma A +\phi_+)+\frac{b^2r^2}4 (b\sigma B +\phi_-)+\NLm,
\end{align*}
we are reduced to solve the system
\begin{equation*}\left\{\begin{aligned}
\Lp \phi_+ &= \FFp(\phi_+,\phi_-)\\
\Lm \phi_- &= \FFm(\phi_+,\phi_-).
\end{aligned}\right.\end{equation*}
We work in the following complete metric space
\[ E_K=\left\{(\phi_+,\phi_-):K\to \R^2 \mbox{ is continuous and satisfies $\|(\phi_+,\phi_-)\|_K\leq 1$}\right\}\\
\]
equipped with the distance associated to the norm
\[
\|(\phi_+,\phi_-)\|_K=\max\left\{ b^{-\frac 13} \Np(\phi_+);b^{-\frac 54} \sigma ^{-1}\Nm(\phi_-) \right\}.
\]
We look for a fixed point to the application
\[
\Gamma_K:(\phi_+,\phi_-)\in E_K \mapsto (\Hp(\FFp(\phi_+,\phi_-)),\Hm(\FFm(\phi_+,\phi_-))).
\]
By the Banach Fixed-Point Theorem, we only have to show that $\Gamma_K$ maps $E_K$ to itself and is a contraction on
$E_K$ for the norm $\|\cdot\|_K$.

 Let $(\phi_+,\phi_-)$, $(\tilde \phi_+,\tilde \phi_-)\in E_K$.
From~\eqref{on:Q}, \eqref{on:bb}, \eqref{on:gamma}, we observe the following estimates
 concerning the terms in the definition of $P$ in \eqref{def:P}, for $b$ small, on $K$,
\begin{equation}\label{all:small}\begin{aligned}
&|\gamma A|\leq C b^{\frac 16} \ee^{-\frac{2}{\sqrt{b}}} (1+r)^{-\frac{d-1}2} \ee^r
\leq C b^{\frac 16} Q\ll Q ,\\
&|\phi_+|+|\tilde\phi_+|\leq C b^\frac 13 Q \ll Q,\\
&b|\sigma B|+|\phi_-|+|\tilde\phi_-|\leq C b\sigma (1+r)^{-\frac{d-1}2} \ee^r
\leq C \ee^{-\frac\pi b+\frac 2{\sqrt{b}}} Q \ll Q.
\end{aligned}\end{equation}

Now, we estimate the terms in $\Gamma_K(\phi_+,\phi_-)$.
First, by~\eqref{e:Np}, $r_K=b^{-\frac 12}$ and next~\eqref{all:small}, one has
\begin{align*}
\Np\left(\Hp\left(\frac{b^2r^2}4 ( Q+\gamma A+\phi_+ )\right)\right)
&\leq C b^{\frac 12} \left( \Np(Q)+|\gamma|\Np(A)+\Np(\phi_+)\right)\\
&\leq C b^{\frac 12} \left(1+b^{\frac 16}+b^\frac 13\right)
\leq C b^{\frac 12}.
\end{align*}
Using also~\eqref{fut1}, for $b$ small enough, one sees that
\begin{align*}
\Np\left(\Hp\left(b\sigma (b\sigma B +\phi_-)
\right)\right)
&\leq C b^{\frac 12}\sigma \left(b\sigma \Np(B)+\Np(\phi_-) \right)\\
&\leq C b^{\frac 12} \sigma \left( \ee^{-\frac \pi b+\frac 2{\sqrt{b}}}
+\ee^{\frac 2{\sqrt{b}}} \Nm(\phi_-)\right)
\leq C \ee^{-\frac{\pi}{b}}.
\end{align*}
Second, by \eqref{e:Nm} and \eqref{all:small},
\begin{align*}
\Nm\left(\Hm\left(b\sigma (\gamma A +\phi_+)\right)\right)
&\leq C b^{\frac 12} \sigma |\gamma|\Nm(A) + Cb \sigma \Np(\phi_+)\leq C b^\frac 43 \sigma 
\end{align*}
and
\begin{equation*}
\Nm\left(\Hm\left(\frac{b^2r^2}4 (b\sigma B +\phi_-)\right)\right)
\leq C b^{\frac 12} \left( b\sigma \Nm(B)+\Nm(\phi_-)\right)
\leq C b^{\frac 32}\sigma .
\end{equation*}

Third, to treat the terms $\NLp$ and $\NLm$ in the definition of $\FFp$ and $\FFm$, we need the following lemma.
\begin{lemma}
If $(\phi_+,\phi_-)\in E_K$ then
\begin{equation}\label{SP1}\begin{aligned}
|\NLp|&\leq C \left(|\gamma A| +|\phi_+|\right)\left(|\gamma A| +|\phi_+|+|b \sigma B|+|\phi_-|\right) Q^{p-2}, \\
|\NLm|&\leq C \left(|b \sigma B|+|\phi_-|\right)\left(|\gamma A| +|\phi_+|+|b \sigma B|+|\phi_-|\right) Q^{p-2} .
\end{aligned}\end{equation}
Moreover, if $(\tilde\phi_+,\tilde\phi_-)\in E_K$ and $\tilde N_\pm$ are defined by
\begin{align*}
&\tilde P= (Q+\gamma A +\tilde\phi_+ ) + \ii ( b\sigma B +\tilde \phi_- ),\\
&\tNLp + \ii\,\tNLm=
|\tilde P|^{p-1} \tilde P -\left\{Q^p+ p Q^{p-1} (\gamma A +\tilde\phi_+ ) 
+ \ii Q^{p-1} ( b\sigma B + \tilde\phi_- )\right\},
\end{align*}
then
\begin{equation}\label{SP3}\begin{aligned}
|\NLp-\tNLp|&\leq C |\phi_+-\tilde \phi_+|
\left(|\gamma A| +|b \sigma B|+|\phi_+|+|\phi_-|+|\tilde \phi_+|+|\tilde \phi_-|\right) Q^{p-2} \\
&\quad + C |\phi_- -\tilde \phi_-|\left(|\gamma A| +|b \sigma B|+|\phi_+|+|\phi_-|+|\tilde \phi_+|+|\tilde \phi_-|\right) Q^{p-2},\\
|\NLm-\tNLm|&\leq C |\phi_+-\tilde \phi_+|
\left( |b \sigma B|+| \phi_-|+|\tilde \phi_-|\right) Q^{p-2}\\
&\quad +C |\phi_- -\tilde \phi_-|
\left(|\gamma A| +|b \sigma B|+|\phi_+|+|\phi_-|+|\tilde \phi_+|+|\tilde \phi_-|\right) Q^{p-2}.
\end{aligned}\end{equation}
\end{lemma}
\begin{proof}
For any $z\in \C$ with $|z|\ll Q$, a Taylor expansion yields
\begin{multline}\label{SP2} 
|Q+z|^{p-1} (Q+z) = Q^p + p Q^{p-1}(\Re z)+\ii Q^{p-1} (\Im z)\\
 +(\Re z) O_{\R}(|z| Q^{p-2})+\ii (\Im z) O_{\R}(|z| Q^{p-2}).
\end{multline}
(The notation $g=O_\R(f)$ means that $g$ is real-valued and $|g|\leq C f$).
Applying those estimates to $z=\gamma A+\phi_++\ii (b\sigma B+\phi_-)$ (note that $|z|\ll Q$ by~\eqref{all:small}), we obtain~\eqref{SP1}.

Let $Z$ be a complex-valued function close to $Q$ such that $C_1 Q\leq |Z|\leq C_2Q$ on $K$
for $C_1,C_2>0$,
and let $z\in \C$ be such that $|z|\ll Q$ on $K$.
Using~\eqref{SP2} with $Q$ replaced
by $|Z|$ and $z$ replaced by $z\frac{\bar Z}{|Z|}$, we compute
\begin{align*}
|Z+z|^{p-1}(Z+z)
&=\frac{Z}{|Z|}\left| |Z|+z\frac{\bar Z}{|Z|}\right|^{p-1} \left( |Z|+z\frac{\bar Z}{|Z|}\right)\\
&=|Z|^{p-1}Z+p\Re (z\bar Z)Z|Z|^{p-3}
+\ii \Im (z \bar Z) Z|Z|^{p-3}\\
&\quad +\Re (z\bar Z) Z\, O_\R(|z| |Z|^{p-4})
+\ii \Im (z\bar Z) Z \, O_\R(|z||Z|^{p-4}).
\end{align*}
Define
\begin{align*}
\triangle
&=|Z+z|^{p-1}(Z+z)-\left\{|Z|^{p-1}Z+p(\Re z)Q^{p-1}+\ii(\Im z)Q^{p-1}\right\}\\
&=p\left\{\Re(z\bar Z)Z|Z|^{p-3}-\Re(z)Q^{p-1}\right\}+\ii\left\{\Im(z\bar Z) Z|Z|^{p-3}-\Im(z)Q^{p-1}\right\}\\
&\quad+\Re(z\bar Z) Z \, O_\R(|z| |Z|^{p-4})+\ii \Im (z\bar Z) Z\, O_\R(|z||Z|^{p-4})\\
&=\triangle_1+\triangle_2+\triangle_3+\triangle_4.
\end{align*}
By direct computations, we observe
\begin{align*}
\Re(z\bar Z) Z
&=(\Re z)|\Re Z|^2+(\Im z)(\Re Z)(\Im Z)+\ii\left\{(\Re z)(\Re Z)(\Im Z)+(\Im z)|\Im Z|^2\right\},\\
\Im(z\bar Z) Z
&=(\Im z)|\Re Z|^2-(\Re z)(\Re Z)(\Im Z)+\ii\left\{(\Im z)(\Re Z)(\Im Z)-(\Re z)|\Im Z|^2\right\}.
\end{align*}
Therefore,
\begin{align*}
p^{-1} \triangle_1
&=(\Re z)\left\{|\Re Z|^2 |Z|^{p-3}-Q^{p-1}\right\}+(\Im z)(\Re Z)(\Im Z)|Z|^{p-3}\\
&\quad+\ii\left\{(\Re z)(\Re Z)(\Im Z)+(\Im z)|\Im Z|^2\right\}|Z|^{p-3}.
\end{align*}
Using $|Z|\le C Q$, $\Im Z=O_\R(|Z-Q|)$ and $|\Re Z|^2 |Z|^{p-3}-Q^{p-1}=O_\R(|Z-Q|Q^{p-2})$, it follows that
\begin{align*}
|\Re \triangle_1|&\le C |z| |Z-Q|Q^{p-2} ,\\
|\Im \triangle_1|&\le C |\Re z| |\Im Z| Q^{p-2}+C|\Im z| |\Im Z|^2Q^{p-3}.
\end{align*}
Similarly, we obtain
\begin{align*}
\triangle_2
&=\ii (\Im z)\left\{|\Re Z|^2 |Z|^{p-3}-Q^{p-1}\right\}-\ii (\Re z)(\Re Z)(\Im Z)|Z|^{p-3}\\
&\quad-\left\{(\Im z)(\Re Z)(\Im Z)-(\Re z)|\Im Z|^2\right\}|Z|^{p-3}.
\end{align*}
and thus
\begin{align*}
|\Re \triangle_2|&\le C |z| |Z-Q|Q^{p-2} ,\\
|\Im \triangle_2|&\le C |\Im z| |Z-Q| Q^{p-2}+C|\Re z| |\Im Z|Q^{p-2}.
\end{align*}
We estimate $\triangle_3$ and $\triangle_4$ as follows
\begin{align*}
|\Re \triangle_3|&\le C |z|^2 Q^{p-2} ,\quad |\Im \triangle_3|\le C |z|^2 |\Im Z| Q^{p-3},\\
|\Re \triangle_4|&\le C |z|^2 Q^{p-2} ,\quad |\Im \triangle_4|\le C|z||\Im z| Q^{p-2} + C |z|^2|\Im Z|Q^{p-3}.
\end{align*}
Combining those estimates, we conclude
\begin{equation}\label{pour:cN}\begin{aligned}
|\Re \triangle|&\le C|z||Z-Q| Q^{p-2}+C|z|^2 Q^{p-2}, \\
|\Im \triangle|&\leq C|\Im z| |Z-Q| Q^{p-2} + C |z| |\Im Z| Q^{p-2}+C|\Im z||z| Q^{p-2}.
\end{aligned}\end{equation}
We observe that $(\NLp+\ii\NLm)-(\tNLp+\ii\tNLm)$ is equal to $\triangle$ with the following choices
\begin{align*}
Z&=\tilde P=(Q+\gamma A +\tilde\phi_+ ) + \ii ( b\sigma B + \tilde\phi_- ),\\
z&=P-\tilde P=(\phi_+-\tilde\phi_+)+\ii (\phi_- -\tilde\phi_-).
\end{align*}
Estimates~\eqref{SP3} thus follow from~\eqref{pour:cN}.
\end{proof}

Observe from~\eqref{all:small} that
\begin{equation}\label{SP4}
\|\gamma A\|_{L^\infty}+\|b \sigma B\|_{L^\infty}+\|\phi_-\|_{L^\infty}+\|\tilde\phi_-\|_{L^\infty}\leq C \ee^{-\frac 1{\sqrt{b}}},\quad
|\phi_+|+|\tilde\phi_+|\leq C b^{\frac 13} Q.
\end{equation}
Set $\bar p=\min(p,2)>1$.
It follows from~\eqref{SP1}, \eqref{all:small} and then \eqref{SP4} that
\begin{align*}
|\NLp| 
&\leq C \left(|\gamma A| +|\phi_+|\right)\left(|\gamma A| +|b \sigma B|+|\phi_-|\right) Q^{p-2}
+C|\phi_+|^2 Q^{p-2}\\
&\leq C \left(|\gamma A| +|\phi_+|\right)\left(|\gamma A| +|b \sigma B|+|\phi_-|\right)^{\bar p-1}
+C|\phi_+|^2 Q^{p-2}\\
&\leq C \ee^{-\frac {\bar p-1}{\sqrt{b}}} \left(|\gamma A| +|\phi_+|\right)+Cb^{\frac 13}|\phi_+|Q^{p-1} .
\end{align*}
Therefore, using~\eqref{e:Np} and \eqref{SP4}, it holds, for $b$ small enough,
\begin{align*}
\Np(\Hp(\NLp))
&\leq C b^{-\frac 12} \ee^{-\frac {\bar p-1}{\sqrt{b}}}
\left( \Np(\gamma A) + \Np(\phi_+)\right)+C b^{\frac 13}\Np((1+r^2)Q^{p-1}\phi_+)\\
&\leq C\ee^{-\frac {\bar p-1}{2\sqrt{b}}} +C b^{\frac{1}3} \Np(\phi_+)
\leq C b^{\frac 23}.
\end{align*}
Again, it follows from~\eqref{SP1}, \eqref{all:small} and then \eqref{SP4} that
\begin{align*}
|\NLm| &\leq 
C \left(|b \sigma B|+|\phi_-|\right)\left(\left(|\gamma A| +|b \sigma B|+|\phi_-|\right)^{\bar p-1} +|\phi_+|Q^{p-2}\right)\\
&\leq C \left(|b \sigma B|+|\phi_-|\right)\left(\ee^{-\frac {\bar p-1}{\sqrt{b}}} +b^\frac 13 Q^{p-1}\right).
\end{align*}
Therefore, using~\eqref{e:Nm} and~\eqref{all:small}, it holds
\begin{align*}
\Nm(\Hm(\NLm)) & \leq C b^{-\frac 12} \ee^{-\frac {\bar p-1}{\sqrt{b}}} \left(b \sigma \Nm( B)+\Nm(\phi_-)\right) \\
&\qquad +b^{\frac 13} \left(b \sigma \Nm((1+r^2) Q^{p-1} B)+\Nm((1+r^2) Q^{p-1} \phi_-)\right)\\
&\leq b^{\frac 13} \left(b \sigma \Nm(B)+\Nm(\phi_-)\right)
 \leq Cb^{\frac 43} \sigma.
\end{align*}
Gathering all the previous estimates, we have proved that
\begin{equation}\label{image}
\|\Gamma_K(\phi_+,\phi_-)\|_K=
\|(\Hp(\FFp(\phi_+,\phi_-)),\Hm(\FFm(\phi_+,\phi_-)))\|_K
\leq C b^{\frac 1{12}} .
\end{equation}
In particular, for $b$ small enough, $\Gamma_K$ maps $E_K$ to itself.

We turn to the estimate of $\Gamma_K(\phi_+,\phi_-)-\Gamma_K(\tilde\phi_+,\tilde\phi_-)$.
First, we see that
\[
\FFp(\phi_+,\phi_-)-\FFp(\tilde \phi_+,\tilde \phi_-)
=\frac{b^2r^2}4 (\phi_+ -\tilde \phi_+) + b\sigma (\phi_- - \tilde \phi_-) + (\NLp-\tilde N_+).
\]
It follows from the previous arguments (using~\eqref{fut1} and~\eqref{e:Np}) that
\begin{align*}
&\Np\left(\Hp\left(\frac{b^2r^2}4 (\phi_+ -\tilde \phi_+)\right)\right)\leq Cb^{\frac 12} \Np(\phi_+-\tilde \phi_+),\\
&\Np \left(\Hp\left(b\sigma (\phi_- -\tilde \phi_-) \right) \right)
\leq Cb^{\frac 12} \sigma \ee^{\frac 2{\sqrt{b}}}\Nm(\phi_- -\tilde \phi_-).
\end{align*}
Next, from~\eqref{SP3} and~\eqref{SP4}, it follows that
\begin{align*}
|\NLp-\tNLp|&\leq C\left( \ee^{-\frac {\bar p-1}{\sqrt{b}}} + |\phi_+| Q^{p-2} +|\tilde \phi_+| Q^{p-2}\right) \left(|\phi_+-\tilde \phi_+|+|\phi_- -\tilde \phi_-|\right)\\
&\leq C \left( \ee^{-\frac {\bar p-1}{\sqrt{b}}} +b^{\frac 13} Q^{p-1}\right) \left(|\phi_+-\tilde \phi_+|+|\phi_- -\tilde \phi_-|\right).
\end{align*}
Thus, proceeding as before, using \eqref{e:Np} and then~\eqref{fut1}, one obtains
\begin{align*}
\Np(\Hp(\NLp-\tNLp))
&\leq C \left (b^{-\frac 12} \ee^{-\frac {\bar p-1}{\sqrt{b}}}+b^{\frac{1}3}\right)
\left( \Np(\phi_+-\tilde\phi_+)+\Np(\phi_- -\tilde \phi_-)\right)\\
&\leq C b^{\frac{1}3}\left( \Np(\phi_+-\tilde\phi_+)+\ee^{\frac 2{\sqrt{b}}}\Nm(\phi_- -\tilde \phi_-)\right).
\end{align*}
In particular,
\[
b^{-\frac 13}\Np(\Hp(\NLp-\tNLp))
\leq C b^{\frac 1{3}} \|(\phi_+,\phi_-)-(\tilde\phi_+,\tilde\phi_-)\|_K.
\]
Last, we observe that
\[
\FFm(\phi_+,\phi_-)-\FFm(\tilde \phi_+,\tilde \phi_-)
=- b\sigma (\phi_+-\tilde \phi_+) +\frac{b^2r^2}4 (\phi_- -\tilde \phi_-)+\NLm-\tNLm.
\]
The estimates 
\begin{align*}
&\Nm\left(\Hm\left(b\sigma(\phi_+ -\tilde \phi_+\right)\right) \leq Cb \sigma \Np(\phi_+-\tilde \phi_+),\\
&\Nm \left(\Hm\left (\frac{b^2r^2}4 (\phi_- -\tilde \phi_-) \right)\right )
\leq Cb^{\frac 12} \Nm(\phi_- -\tilde \phi_-) 
\end{align*}
follow from~\eqref{e:Nm} and previous arguments.
Next, from~\eqref{SP3}, \eqref{all:small} and~\eqref{SP4}, it follows that
\begin{equation*}
|\NLm-\tNLm|\leq C |\phi_+-\tilde \phi_+|\left(b \sigma H^{-1} Q^{p-2}\right) +C|\phi_- -\tilde \phi_-| \left( \ee^{-\frac {\bar p-1}{\sqrt{b}}} + b^{\frac 13}Q^{p-1} \right) .
\end{equation*}
Note that \eqref{e:Nm} and the definitions of $\Np$ and $\Nm$ imply
\begin{align*}
\Nm(\Hm(H^{-1} Q^{p-2} (\phi_+-\tilde \phi_+))
&\leq C \Nm((1+r^2) H^{-1} Q^{p-2}(\phi_+-\tilde \phi_+))\\
&\leq C\Np((1+r^2) Q^{p-1}(\phi_+-\tilde \phi_+))
\leq C\Np(\phi_+-\tilde \phi_+).
\end{align*}
Using \eqref{e:Nm}, proceeding as before,
\[
\Nm\left(\left( \ee^{-\frac {\bar p-1}{\sqrt{b}}} + b^{\frac 13}Q^{p-1} \right)(\phi_- -\tilde \phi_-) \right)
\leq C b^{\frac 13} \Nm(\phi_- -\tilde \phi_-).
\]
Thus, we obtain
\begin{equation*}
\Nm(\Hm(\NLm-\tNLm))
\leq C b\sigma \Np(\phi_+-\tilde\phi_+)+ C b^{\frac 13} \Nm(\phi_- -\tilde \phi_-).
\end{equation*}
In particular,
\[
b^{-\frac 54} \sigma ^{-1} \Nm(\Hm(\NLm-\tNLm))
\leq C b^{\frac 1{12}} \|(\phi_+,\phi_-)-(\tilde\phi_+,\tilde\phi_-)\|_K.
\]
Gathering the previous estimates, we have proved
\begin{equation*}
\|\Gamma_K(\phi_+,\phi_-)-\Gamma_K(\tilde\phi_+,\tilde\phi_-)\|_K
\leq C b^{\frac 1{12}} \|(\phi_+,\phi_-)-(\tilde\phi_+,\tilde\phi_-)\|_K,
\end{equation*}
which implies that the map $\Gamma_K:E_K\to E_K$ is a contraction for the norm ${\|\cdot\|_K}$
provided that $b$ is small enough.
The Banach Fixed-Point Theorem shows the existence of a unique fixed point $(\phi_+,\phi_-)$ of $\Gamma_K$ in $E_K$.

Note that
\begin{equation*}
\phi_+' = [\Hp(\FFp(\phi_+,\phi_-))]',\quad
\phi_-'= [\Hm(\FFm(\phi_+,\phi_-)))]'.
\end{equation*}
Since~\eqref{e:DNp} and \eqref{e:DNm} are the same estimates for $[\Hp(f)]'$ and $[\Hm(f)]'$ as~\eqref{e:Np} and \eqref{e:Nm} for $\Hp(f)$ and $\Hm(f)$, the bound $\|(\phi_+',\phi_-')\|_K \leq C b^{\frac 1{12}}\leq 1$ follows readily from the proof of~\eqref{image}. 

By \eqref{def:P}, we obtain a solution $P$ of \eqref{eq:Pb} on $K$, whose
estimates \eqref{Pint1}-\eqref{Pint4}
 at $r=r_K$ follow directly from the asymptotic behaviors of $Q$, $A$ and $B$ as $r\to \infty$ as described in~\eqref{on:Q}, Lemmas~\ref{le:7} and~\ref{le:10}, and the fact that
$\|(\phi_+,\phi_-)\|_K\leq 1$ and $\|(\phi_+',\phi_-')\|_K \leq 1$.
Note that these estimates, the definitions of $\Np$, $\Nm$ and the definition of $P$ in \eqref{def:P} also show that $\|P-Q\|_{\dot H^1(K)}\leq C b^{\frac 1{12}}$.

The continuity statement follows from usual arguments.
\end{proof}

\section{Matching}\label{S:matching}
The goal of this final section is to prove the following general existence result, obtained by matching
at the point $r_K=b^{-\frac 12}$ the solution $P_\ext$ constructed in Proposition~\ref{pr:Uext} and the solution $P_\itt$ constructed in Proposition~\ref{pr:Pint}.
In Remark~\ref{rk:match}, we have already pointed out the main analogies between the behaviors of $P_\ext(r_K)$
\eqref{Pext1}-\eqref{Pext4} and of $P_\itt(r_K)$, \eqref{Pint1}-\eqref{Pint4}.
\begin{theorem}\label{th:2}
Let $d\geq 1$ and $p_*< \bar p$.
There exists $\sigma_0>0$ such that for any $\sigma\in (0,\sigma_0)$ and $p\in [p_*,\overline p]$, there exist $b$, $\rho$, $\gamma$ and $\theta$ satisfying
\eqref{on:b}, \eqref{on:rho}, \eqref{on:gamma} and \eqref{on:theta}, such that the solutions $P_\ext[\sigma,b,\rho]$ of~\eqref{eq:Pb} on $I\cup J$ given by Proposition~\ref{pr:Uext} and $P_\itt[\sigma,b,\gamma]$ of~\eqref{eq:Pb} on $K$ 
given by Proposition~\ref{pr:Pint} satisfy the following \emph{matching conditions:}
\begin{equation*}
P_\itt(r_K)= \ee^{\ii \theta} P_\ext(r_K)\quad \mbox{and}\quad
P_\itt'(r_K)= \ee^{\ii \theta} P_\ext'(r_K) .
\end{equation*}
In particular, the function $P$ defined on $[0,\infty)$ by
\[
P(r)=
\begin{cases} \ee^{\ii \theta} P_\ext(r) & \text{for $r\in I\cup J$,}\\
P_\itt(r) & \text{for $r\in K$,}
\end{cases}
\]
is a $\mathcal C^2$ solution of \eqref{eq:Pb} on $[0,\infty)$ satisfying the asymptotics: for $r$ large,
\begin{align}
|P(r)| & = |\rho| r^{-\frac d2+\sigma} (1+O(b^{-3}r^{-2})),\label{eq:th2.1}\\
 P'(r)-\ii \frac{br}2 P(r) & = O( \rho b^{-1} r^{-\frac d2-1+\sigma}).\label{eq:th2.2}
\end{align}
\end{theorem}
\begin{proof}
Let $\theta$ satisfy~\eqref{on:theta}.
Matching \eqref{Pext1}-\eqref{Pext4} with \eqref{Pint1}-\eqref{Pint4}, removing some multiplicative factors, using  
$
\cos \theta = 1 + O(\theta_\sigma^2)$ and
$\sin \theta = \theta + O(\theta_\sigma^3)$,
 we obtain the following four equivalent conditions:
\begin{align}
&\frac{\rho b^{\frac 12}} {\sqrt{2}}
\exp\left(\frac{\pi}{2b}\right) (1+O(b^{\frac 14}))
= \kappa (1+  O(b^{\frac 13}))+\kappa_A\gamma \exp\left(\frac2{\sqrt{b}}\right) ,\label{Pmatch1}\\
&\frac{\rho b^{\frac 12 }}{\sqrt{2}}\exp\left(\frac{\pi}{2b}\right) (1+O(b^{\frac 14}))
= \kappa (1+  O(b^{\frac 13}))
- \kappa_A \gamma \exp\left(\frac2{\sqrt{b}}\right) ,\label{Pmatch2}\\
&\frac\rho{2\sqrt{2}} \exp\left(-\frac{\pi}{2b}\right) (1+O(b^{\frac 14}))
+\frac{\theta \rho}{\sqrt{2}}\exp\left(\frac\pi{2b}\right) \exp\left(-\frac2{\sqrt{b}}\right) (1+O(b^{\frac 14}))\nonumber
\\&\qquad= \kappa_B \sigma b^{\frac 12} (1+  O(b^{\frac 14})), \label{Pmatch3}\\
&\frac\rho{2\sqrt{2}} \exp\left(-\frac{\pi}{2b}\right) (1+O(b^{\frac 14}))
-\frac{\theta \rho}{\sqrt{2}}\exp\left(\frac\pi{2b}\right) \exp\left(-\frac2{\sqrt{b}}\right) (1+O(b^{\frac 14}))\nonumber
\\&\qquad = \kappa_B \sigma b^{\frac 12} (1+  O(b^{\frac 14})) \label{Pmatch4} ,
\end{align}
where the notation $O(b^\frac14)$ (and similarly $O(b^\frac 13)$) stands for different real-valued continuous functions of
the parameters $b$, $\rho$, $\gamma$ and $\theta$, that are bounded by $C b^{\frac 14}$.
Recall from Lemmas~\ref{le:7} and~\ref{le:10} that $\kappa_A\neq 0$ and $\kappa_B\neq 0$.

We observe that \eqref{Pmatch1}-\eqref{Pmatch2} are equivalent to
\begin{align}
\rho &= \sqrt{2} \kappa b^{-\frac 12}\exp\left(-\frac{\pi}{2b}\right) (1+ O(b^{\frac 14})),\label{eq:lastrho}\\
\gamma&=O\left(b^{\frac 14} \exp\left(-\frac2{\sqrt{b}}\right)\right).\label{eq:lastgamma}
\end{align}
Inserted in \eqref{Pmatch3}-\eqref{Pmatch4}, we also obtain, using notation from \eqref{eq:kappaB}
and \eqref{on:sigma},
\begin{align}
\sigma&= \frac{\kappa^2 }{N_c}b^{-1}\exp\left(-\frac{\pi}{b}\right) (1+ O(b^{\frac 14})),\label{on:sigmab}\\
\theta&= O\left(b^{\frac 14} \exp\left(-\frac\pi b\right)\exp\left(\frac2{\sqrt{b}}\right)\right).
\label{eq:lasttheta}\end{align}
Define
\begin{equation*}
\tilde b= \frac{b-b_\sigma}{b_\sigma^{\frac{13}6}}.
\end{equation*}
From the definition of $b_\sigma$ in \eqref{on:sigma} and \eqref{on:sigmab}, we obtain
(see also \eqref{on:bb})
\[
\left| \exp \left(\frac \pi b-\frac\pi{b_\sigma}\right) - 1\right| \leq C b_\sigma^{\frac 14}\quad
\mbox{and so}\quad |\tilde b|\leq C b_\sigma^{\frac 1{12}}.
\]
Define also
\begin{equation*}
\tilde \rho = \frac{\rho-\rho_\sigma}{\rho_\sigma},\quad
\tilde \gamma=\frac{\gamma}{\gamma_\sigma},\quad \tilde \theta=\frac{\theta}{\theta_\sigma}.
\end{equation*}
The estimates \eqref{eq:lastrho}, \eqref{eq:lastgamma} and \eqref{eq:lasttheta} imply respectively
\begin{equation*}
|\tilde \rho|\leq C b_\sigma^{\frac 14},\quad |\tilde\gamma|\leq Cb_\sigma^{\frac 1{12}},\quad |\tilde \theta|\leq Cb_\sigma^{\frac 1{12}}.
\end{equation*}

Therefore, the matching relations \eqref{Pmatch1}-\eqref{Pmatch4} rewrite equivalently as
\begin{equation}\label{eq:pfPhi}
(\tilde b,\tilde \rho,\tilde \gamma,\tilde \theta)=\Phi(\tilde b,\tilde \rho,\tilde \gamma,\tilde \theta),
\end{equation}
where the function
$\Phi:[-\frac 12,\frac12]^4 \to \R^4$
is continuous and satisfies the bound
\begin{equation*}
\left|\Phi(\tilde b,\tilde \rho,\tilde \gamma,\tilde \theta)\right|\leq C b_\sigma^{\frac 1{12}}
\quad \mbox{for all } (\tilde b,\tilde \rho,\tilde \gamma,\tilde \theta)\in\left[-\frac 12,\frac12\right]^4.
\end{equation*}
In particular, for $\sigma$ small enough, 
$\Phi\big(\left[-\frac 12,\frac12\right]^4\big)\subset\left[-\frac 12,\frac12\right]^4$.
By the Brouwer Fixed-Point Theorem, there exists at least a fixed-point of the function $\Phi$ in $\left[-\frac 12,\frac12\right]^4$. This fixed-point $(\tilde b,\tilde \rho,\tilde \gamma,\tilde\theta)$ 
 satisfies \eqref{eq:pfPhi} and the estimate
\[
|\tilde b|+|\tilde \rho|+|\tilde \gamma|+|\tilde \theta|\leq C b_\sigma^{\frac 1{12}}.
\]
The corresponding values of $b$, $\rho$, $\gamma$ and $\theta$ satisfy~\eqref{Pmatch1}-\eqref{Pmatch4}.
This completes the matching argument. The second part of the statement of  Theorem~\ref{th:2} follows from standard arguments.
\end{proof}

\begin{proof}[Proof of Theorem \ref{th:1}]
We apply Theorem~\ref{th:2} with $p>p_*$ close enough to $p_*=1+\frac4d$ and $\sigma=s_c=\frac d2-\frac 2{p-1}$.
Let $\Psi$ be defined by \eqref{PsiP} where $P$ is given by Theorem~\ref{th:2}.
The fact that $\Psi$ satisfies \eqref{eq:Qb}
is easily checked from the equation of $P$. The facts that $\Psi\in \dot H^1$, $\|\Psi-Q\|_{\dot H^1}=o(1)$
as $p\to p_*$, and the asymptotic behavior of $\Psi(r)$ as $r\to \infty$, follow from the property \eqref{eq:closeint} of $P_\itt$ in Proposition~\ref{pr:Pint}  and the asymptotics \eqref{eq:th2.1}-\eqref{eq:th2.2}.
See also Remark~\ref{rk:h1dot}.
\end{proof}

\begin{remark}
We believe that the regularity of the solutions $\Psi$ of Theorem~\ref{th:1} with respect to the exponent $p$ can be obtained from our computations and standard arguments.
However, it would require significant additional work and we have not pursued this issue.
\end{remark}


\begin{thebibliography}{10}

\bibitem{bizon} Biernat, P.; Bizo{\'n}, P., 
Shrinkers, expanders, and the unique continuation beyond generic blowup in the heat flow for harmonic maps between spheres, \emph{Nonlinearity} \textbf{24} (2011), 2211--2228.

\bibitem{BCR} Budd, C. J.; Chen, S.; Russell, R.D.,
New self-similar solutions of the nonlinear Schr\"odinger equation with moving mesh computations,
\emph{J. of Comp. Physics} \textbf{152} (1999), 756--789.

\bibitem{Budd} Budd, C. J.,
Asymptotics of multibump blow-up self-similar solutions of the nonlinear Schr\"odinger equation,
\emph{SIAM J. Appl. Math.} \textbf{62} (2002), 801--830.

\bibitem{Caz} Cazenave, T., 
\emph{Semilinear Schr\"odinger Equations}, 
Courant Lecture Notes in Mathematics 10, New York University.
CIMS, New York; AMS, Providence, RI (2003).

\bibitem{CGNT} Chang S.-M.;   Gustafson S.; Nakanishi K.; Tsai T.-P.,
Spectra of linearized operators for NLS solitary waves,
\emph{SIAM J. Math. Anal.} \textbf{39} (2007/08), 1070--1111. 

\bibitem{CRS} Collot, C.; Rapha\"el, P.; Szeftel, J., 
On the stability of self similar blow up for the energy super critical heat equation,
to appear in \emph{Memoirs AMS}.

\bibitem{Fedoryuk} Fedoryuk, M. V.,
\emph{Asymptotic analysis.}
Springer-Verlag, Berlin, 1993. Linear ordinary differential equations, 
Translated from the Russian by A. Rodick.

\bibitem{FGW}
Fibich, G.; Gavish, N.; Wang, X-P.,
Singular ring solutions of critical and supercritical non-linear Schr\"odinger equations,
\emph{Physica D} \textbf{231} (2007),  55--86.

\bibitem{GV} Ginibre, J.; Velo, G.,
On a class of nonlinear Schr\"odinger equations. I. The Cauchy problem, general case, 
\emph{J. Funct. Anal.} \textbf{32} (1979),  1--32. 

\bibitem{Godet} Godet, N., 
Blow up on a curve for a nonlinear Schr\"odinger equation on Riemannian surfaces,
\emph{Dynamics of PDE}, \textbf{10} (2013), 99--155.

\bibitem{JP} Johnson, R.; Pan, X., 
On an elliptic equation related to the blow-up phenomenon in the nonlinear Schrödinger equation,
\emph{Proceedings of the Royal Society of Edinburgh: Section A Mathematics}, \textbf{123} (1993), 763--782. %doi:10.1017/S030821050003095X

\bibitem{KW} Kavian, O.; Weissler, F. B.,
Self-similar solutions of the pseudo-conformally invariant nonlinear Schr\"odinger equation,
\emph{Michigan Math. J.}, \textbf{41} (1994), 151--173. %doi:10.1307/mmj/1029004922. https://projecteuclid.org/euclid.mmj/1029004922

\bibitem{Koch} Koch, H., 
Self-similar solutions to super-critical gKdV,
\emph{Nonlinearity} \textbf{28} (2015), 545--575.

\bibitem{KL} Kopell, N.; Landman, M., 
Spatial structure of the focusing singularity of the nonlinear Schr\"odinger equation: a geometrical analysis,
\emph{SIAM J. Appl. Math.} \textbf{55} (1995), 1297--1323.

\bibitem{LPSS} Landman, M.; Papanicolaou, G.G.; Sulem, C.; Sulem, P.-L., 
Rate of blowup for solutions of the nonlinear Schr\"odinger equation at critical dimension,
\emph{Phys. Rev. A} \textbf{38} (1988), 3837--3843.

\bibitem{LPSS1} Le Mesurier, B.G.; Papanicolaou, G.G.; Sulem, C.; Sulem, P.-L., 
Focusing and multi-focusing solutions of the nonlinear Schr\"odinger equation,
\emph{Phys. D} \textbf{31} (1988), 78--102. 

\bibitem{LPSS2} Le Mesurier, B.G.; Papanicolaou, G.G.; Sulem, C.; Sulem, P.-L., 
Local structure of the self-focusing singularity of the nonlinear Schr\"odinger equation,
\emph{Phys. D} \textbf{32} (1988), 210--226.

\bibitem{lernerbook} Lerner, N., 
\emph{Fonctions classiques.}
Lectures notes, Universit\'e Pierre et Marie Curie, 2017.

\bibitem{MMRkdv3} Martel, Y.; Merle, F.; Rapha\"el, P.,
Blow up for the critical gKdV equation III: exotic regimes,
\emph{Ann. Sc. Norm. Super. Pisa Cl. Sci.} (5) \textbf{14} (2015),  575--631.

\bibitem{MMRkdv2} Martel, Y.; Merle, F.; Rapha\"el, P.,
Blow up for the critical gKdV equation II: minimal mass blow up,
\emph {J. Eur. Math. Soc.} (JEMS) \textbf{17} (2015),   1855--1925.

\bibitem{MMRkdv1} Martel, Y.; Merle, F.; Rapha\"el, P.,
Blow up for the critical gKdV equation I: dynamics near the solitary wave, 
\emph{Acta Math.} \textbf{212} (2014),  59--140. 

\bibitem{MRtwo} Martel, Y.; Rapha\"el, P.,
Strongly interacting blow up bubbles for the mass critical NLS,
\emph{Ann. Sci. Ec. Norm. Sup.} \textbf{51}  (2018), 701--737.

\bibitem{MR1} Merle, F.; Rapha\"el, P.,
Blow up dynamic and upper bound on the blow up rate for critical nonlinear Schr\"odinger equation,
\emph{Ann. Math.} \textbf{161} (2005), 157--222.

\bibitem{MeRa04} Merle, F.; Rapha\"el, P., 
On universality of blow-up profile for $L^2$ critical nonlinear Schr\"odinger equation,
\emph{Invent. Math.} \textbf{156} (2004) 565--672.

\bibitem{MR3} Merle, F.; Rapha\"el, P., 
Sharp lower bound on the blow up rate for critical nonlinear Schr\"odinger equation,
\emph{J. Amer. Math. Soc.} \textbf{19} (2006), 37--90.

\bibitem{MeRa07} Merle, F.; Rapha\"el, P., 
Profiles and quantization of the blow up mass for critical nonlinear Schr\"odinger equation,
\emph{Comm. Math. Phys.} \textbf{253} (2005), 675--704.

\bibitem{MRSring} Merle, F.; Rapha\"el, P.; Szeftel, J., 
Collapsing ring blow up solutions to the $L^2$ super critical NLS, 
\emph{Duke Math. J.} 163 (2014), 369--431.

\bibitem{MRS} Merle, F.; Rapha\"el, P.; Szeftel, J., 
Stable self-similar blow-up dynamics for slightly $L^2$ super-critical NLS equationsn 
\emph{Geom. Funct. Anal.} \textbf{20} (2010), 1028--1071.

\bibitem{P} Perelman, G.,
On the formation of singularities in solutions of the critical nonlinear Schr\"odinger equation, 
\emph{Ann. Henri Poincar\'e} \textbf{2} (2001), 605--673. 

\bibitem{PS} Plecháč, P.; Verk, V.,
On self-similar singular solutions of the complex Ginzburg-Landau equation,
\emph{Comm. Pure Appl. Math.} \textbf{54} (2001), 1215--1242.

\bibitem{RK} Rottschäfer, V.; Kaper, T.J., 
Blowup in the nonlinear Schr\"odinger equation near critical dimension,
\emph{J. Math. Anal. Appl.} \textbf{268} (2002), 517--549.

\bibitem{SulemSulem2} Sulem, C.; Sulem, P.L., 
Focusing nonlinear Schr\"odinger equation and wave-packet collapse,
Proceedings of the Second World Congress of Nonlinear Analysts, Part 2 (Athens, 1996). 
\emph{Nonlinear Anal.} \textbf{30} (1997), 833--844.

\bibitem{SulemSulem} Sulem, C.; Sulem, P.-L., 
\emph{The nonlinear Schr\"odinger equation. Self-focusing and wave collapse.} 
Applied Mathematical Sciences, 139. Springer-Verlag, New York, 1999.

\bibitem{YRZ} Yang, K.; Roudenko, S.;  Zhao, Y.,
Blow-up dynamics in the mass super-critical NLS equations,
Physica D, \textbf{396} (2019), 47--69.

\bibitem{Zak} Zakharov, V. E.,
Collapse of self-focusing Langmuir waves,
\emph{Soviet Physics JETP}, \textbf{35} (1972), 908--914.

\end{thebibliography}
\end{document}